\documentclass[a4paper]{amsart}
\usepackage{amsmath,amsthm,amssymb,latexsym,epic}
\usepackage{graphicx,enumerate,stmaryrd}
\usepackage[all]{xy}

\newtheorem{theorem}{Theorem}
\newtheorem{lemma}[theorem]{Lemma}
\newtheorem{remark}[theorem]{Remark}
\newtheorem{corollary}[theorem]{Corollary}
\newtheorem{proposition}[theorem]{Proposition}

\newtheorem{conjecture}[theorem]{Conjecture}

\newcommand{\tto}{\twoheadrightarrow}

\font\sc=rsfs10
\newcommand{\cC}{\sc\mbox{C}\hspace{1.0pt}}
\newcommand{\cL}{\sc\mbox{L}\hspace{1.0pt}}
\newcommand{\cP}{\sc\mbox{P}\hspace{1.0pt}}
\newcommand{\cA}{\sc\mbox{A}\hspace{1.0pt}}

\font\scc=rsfs7

\newcommand{\ccP}{\scc\mbox{P}\hspace{1.0pt}}
\newcommand{\ccA}{\scc\mbox{A}\hspace{1.0pt}}

\begin{document}

\title[Serre functors for Lie superalgebras]{Serre functors 
for Lie algebras and superalgebras}
\author{Volodymyr Mazorchuk and Vanessa Miemietz}
\date{\today}

\begin{abstract}
We propose a new realization, using Harish-Chandra bimodules, 
of the Serre functor for the BGG category $\mathcal{O}$ 
associated to a semi-simple complex finite dimensional Lie 
algebra. We further show that our realization carries over
to classical Lie superalgebras in many cases. Along the way
we prove that category $\mathcal{O}$ and its parabolic
generalizations for classical Lie superalgebras are categories
with full projective functors. As an application we prove that
in  many cases the endomorphism algebra of the basic 
projective-injective module in (parabolic) category $\mathcal{O}$
for classical Lie superalgebras is symmetric. As a special case
we obtain that in these cases the algebras describing blocks 
of the category
of finite dimensional modules are symmetric. We also compute
the latter algebras for the superalgebra $\mathfrak{q}(2)$.
\end{abstract}

\maketitle


\section{Introduction and description of the results}\label{s0}

The category of finite dimensional modules over a semi-simple complex
finite dimensional Lie algebra is semi-simple and hence 
completely understood. For Lie superalgebras the situation is
rather more complicated. Although many blocks of the  
category of finite dimensional modules over a Lie superalgebra
are still semi-simple, a general block might have infinitely
many simple modules or infinite global dimension. The recent preprint
\cite{BS} gives a combinatorial description of associative algebras
whose module categories are equivalent to blocks of the category of 
finite dimensional $\mathfrak{gl}(m,n)$-modules. This description 
implies many nice properties of these algebras, in particular, it turns 
out that these algebras are symmetric and Koszul (this is proved in
earlier articles of the series). To solve a similar 
problem for other Lie superalgebras is an open and seemingly difficult task.

Understanding the category of finite dimensional modules is the first 
step towards understanding the representation theory of a Lie superalgebra.
The next natural step is to understand the analogue of the BGG category
$\mathcal{O}$ (and its numerous generalizations). Various questions
related to the general theory of highest weight modules over 
finite dimensional Lie superalgebras have been studied  in
\cite{Br2,Br3,Go2,CLW}, see also references therein. However, as in the
case of finite dimensional modules, most of the problems in this
theory are still to be solved.

Among the classical Lie superalgebras the queer Lie superalgebra
$\mathfrak{q}(n)$ stands out in many ways. While it looks
deceptively easy, consisting only of an even and an odd copy
of $\mathfrak{gl}_n$, its representation theory is surprisingly
intricate. One of the roots of this problem is that the classical 
Cartan subalgebra of $\mathfrak{q}(n)$ is non-commutative. Two major
achievements in the representation theory of $\mathfrak{q}(n)$
were a description of characters of simple finite 
dimensional $\mathfrak{q}(n)$-modules given in \cite{PS,PS2},
and a relation of character formulae for $\mathfrak{q}(n)$ to
Kazhdan-Lusztig combinatorics in \cite{Br3}. The latter article 
contains, in particular, conjectural combinatorics for the whole category
$\mathcal{O}$. For ``easy'' blocks this conjecture was proved in \cite{FM}.

Motivated by the ultimate goal of proving the conjectures from 
\cite{Br3} and extending the results of \cite{BS} to 
$\mathfrak{q}(n)$, in the present article we take some first steps 
in this direction. We consider a classical Lie superalgebra
$\mathfrak{g}$ from the list
$\mathfrak{gl}(m,n)$, $\mathfrak{sl}(m,n)$, $\mathfrak{osp}(m,n)$, 
$\mathfrak{psl}(n,n)$, $\mathfrak{q}(n)$, $\mathfrak{pq}(n)$, 
$\mathfrak{sq}(n)$, $\mathfrak{psq}(n)$. We show that a block of the
category of finite dimensional $\mathfrak{g}$-modules, related in a
certain way to a simple strongly typical module which is not stable 
under parity change, is described by a symmetric algebra.
For the first four superalgebras this covers all blocks and for
the last four this covers roughly half of the blocks.

To this end we study Serre functors for the category 
$\mathcal{O}$ and its parabolic generalizations. The Lie algebraic
counterpart of this theory was developed in \cite{MS} and is heavily
based on the combinatorial description of blocks of 
$\mathcal{O}$ for Lie algebras, which does not yet exist for Lie
superalgebras. Therefore we are forced to reprove the main 
results from \cite{MS} in a completely different way. The paper
\cite{MS} establishes the important property that 
the Serre functor naturally commutes with projective functors
(in the sense of \cite{Kh}). This silently uses the fact that
the category $\mathcal{O}$ for Lie algebras is a category with 
full projective functors in the sense of \cite{Kh}. 

We show that, under some natural assumptions, the category 
$\mathcal{O}$ for Lie superalgebras as well as its parabolic 
generalizations are categories with full projective functors. 
This setup involves a choice of the so-called ``dominant object'' 
$\Delta$, which in our case turns out to be a (parabolic) Verma 
module, whose highest weight satisfies certain conditions.
We observe that a convenient way to define functors naturally 
commuting with projective functors is to use Harish-Chandra bimodules. 
For two $\mathfrak{g}$-modules $M$ and $N$ the space
$\mathcal{L}(M,N)$ of all linear maps from $M$ to $N$ which are
locally finite with respect to the adjoint action of the
even part has a natural $\mathfrak{g}$-bimodule structure.
Extending results of \cite{BG} and \cite{MiSo} we show that 
certain direct summands $\mathcal{O}_{\overline{\lambda}}^{\mathfrak{p}}$ 
of the (parabolic) category $\mathcal{O}^{\mathfrak{p}}$ for Lie 
superalgebras are equivalent to categories of Harish-Chandra bimodules.
On the latter we have the usual restricted duality ${}_-{}^{\circledast}$.
As the main result of the paper we prove the following.

\begin{theorem}\label{tmain}
For category $\mathcal{O}_{\overline{\lambda}}^{\mathfrak{p}}$,
the left derived of the functor
$\mathcal{L}({}_-,\Delta)^{\circledast}\otimes_{U(\mathfrak{g})}\Delta$
is a Serre functor on the corresponding category of perfect complexes.
\end{theorem}

Even in the case of Lie algebras this description of the Serre functor
is new. The proof relies heavily on methods developed by Gorelik
in \cite{Go4}, which we adapt to our more general situation. 
As an application we obtain that the endomorphism
algebra of the basic projective-injective module in 
$\mathcal{O}_{\overline{\lambda}}^{\mathfrak{p}}$ is symmetric.
For Lie superalgebras of type $I$ this reproves and strengthens 
\cite[Theorem~3.9.1]{BKN} which says that the endomorphism algebra of 
a basic projective-injective finite dimensional module is weakly symmetric.
We use the latter result to compute the associative algebras describing
blocks of the category of finite dimensional $\mathfrak{q}(2)$-modules.

The article is organized as follows: In Section~\ref{s1} we collect
necessary preliminaries on modules over $\Bbbk$-linear categories,
Serre functors and symmetric algebras. In Section~\ref{s2} 
we recall the theory of categories with full projective functors
following \cite{Kh}.
In Section~\ref{s3} and Section~\ref{s4} we prove the main results 
for Lie algebras and superalgebras, respectively.
Finally, Section~\ref{s5} is devoted to the above mentioned example
of $\mathfrak{q}(2)$.

\vspace{0.5cm}

\noindent
{\bf Acknowledgments.} The first author was partially supported
by the Swedish Research Council. This work was done
during a visit of the second author to Uppsala University, which
was supported by the Faculty of Natural Science of Uppsala University.
The financial support and the hospitality of Uppsala University
are gratefully acknowledged. We thank Maria Gorelik for sharing
with us her ``super''ideas. We are grateful to the referee for
very useful comments and suggestions.

\section{Serre functors for $\Bbbk$-linear categories}\label{s1}

\subsection{Conventions}\label{s1.0}

For an abstract category $\cC$ and two objects $x,y\in\cC$
we denote by $\cC(x,y)$ the set of morphisms from $x$ to $y$
in $\cC$. At the same time, for categories of modules and the
corresponding derived categories, we will
use the usual $\mathrm{Hom}$ notation. In particular, working
with $\cC$-modules (that is functors from $\cC$ to some fixed
category), $\mathrm{Hom}_{\cC}$ will mean morphisms
in the category of $\cC$-modules.

For an abelian category $\mathcal{A}$ we denote by 
$\mathcal{D}^{-}(\mathcal{A})$ and $\mathcal{D}^{b}(\mathcal{A})$
the derived categories of complexes (of objects in $\mathcal{A}$)
bounded from the right and from both sides, respectively. 
If $\mathcal{A}$ is the category of $\cC$-modules for some category
$\cC$, the corresponding derived categories
$\mathcal{D}^{-}(\mathcal{A})$ and $\mathcal{D}^{b}(\mathcal{A})$
will be denoted simply by 
$\mathcal{D}^{-}(\cC)$ and $\mathcal{D}^{b}(\cC)$, respectively 
(and we will never consider categories
of modules over abelian categories).
For a right exact functor $\mathrm{F}$ on $\mathcal{A}$
we denote by $\cL\mathrm{F}$ 
the corresponding left derived functor.

For the rest of the paper we fix an algebraically closed field 
$\Bbbk$ and denote by ${{}_-}^*$ the usual duality 
$\mathrm{Hom}_{\Bbbk}({}_-,\Bbbk)$. 
All categories in the paper are assumed to be $\Bbbk$-linear
(i.e. enriched over $\Bbbk\text{-}\mathrm{mod}$) and all
functors are supposed to be $\Bbbk$-linear and additive.
If not stated otherwise, a functor always means a {\em covariant}
functor.

\subsection{Serre functors}\label{s1.1}

Let $\cC$ be a $\Bbbk$-linear additive category with 
finite dimensional morphism spaces. A {\em Serre functor} on 
$\cC$ is an additive auto-equivalence $\mathrm{F}$ of $\cC$ together 
with isomorphisms
\begin{equation}\label{eq1}
\Psi_{x,y}:\cC(x,\mathrm{F}y)\cong\cC(y,x)^*,
\end{equation}
natural in $x$ and $y$ (see \cite{BK}). If a Serre functor exists,
it is unique (up to isomorphism) and commutes with all 
auto-equivalences of $\cC$.

For example, let $A$ be a finite dimensional associative 
$\Bbbk$-algebra of finite global dimension. Then the 
category $\mathcal{D}^b(A)$ always has a Serre functor,
which is given by the left derived of the {\em Nakayama functor}
$A^*\otimes_A{}_-:A\text{-}\mathrm{mod}\to A\text{-}\mathrm{mod}$
(see \cite{Ha}). Note that we have an isomorphism of functors
$A^*\otimes_A{}_-\cong \mathrm{Hom}_{A}({}_-,A)^*$ (as both are
right exact and agree on the projective generator $A$).
The last example is typical in the sense that to have a Serre 
functor one usually has to extend the original abelian category
$A\text{-}\mathrm{mod}$ to the triangulated category 
$\mathcal{D}^b(A)$ (or some other triangulated category, for example,
the category of perfect complexes, see Subsection~\ref{s1.3}).

\subsection{Infinite dimensional setup}\label{s1.2}

In what follows we will often work with abelian categories having
infinitely many isoclasses of simple objects. Therefore
we will need an ``infinite dimensional'' generalization of the
last example. Consider a small $\Bbbk$-linear category 
$\mathcal{C}$  which satisfies the following assumptions:
\begin{enumerate}[$($I$)$]
\item\label{cond0} $\mathcal{C}$ is {\em basic} in the sense that 
different objects from $\mathcal{C}$ are not isomorphic;
\item\label{cond1} for any $x,y\in \mathcal{C}$ the $\Bbbk$-vector space 
$\mathcal{C}(x,y)$ is finite dimensional;
\item\label{cond2} for any $x\in \mathcal{C}$ there exist only finitely many
$y\in \mathcal{C}$ such that  $\mathcal{C}(x,y)\neq 0$;
\item\label{cond3} for any $x\in \mathcal{C}$ there exist only finitely many
$y\in \mathcal{C}$ such that $\mathcal{C}(y,x)\neq 0$;
\item\label{cond4} for any $x\in \mathcal{C}$ the endomorphism algebra
$\mathcal{C}(x,x)$ is local and basic.
\end{enumerate}
We will call such categories {\em strongly locally finite},
or simply {\em slf-categories}.

A {\em left $\mathcal{C}$-module} is a covariant functor 
$M:\mathcal{C}\to \Bbbk\text{-}\mathrm{Mod}$. A $\mathcal{C}$-module 
$M$ is said  to be {\em finite dimensional} provided that 
$\sum_{x\in \mathcal{C}}\dim M(x)<\infty$.
We denote by  $\mathcal{C}\text{-}\mathrm{mod}$ the category of all
finite dimensional $\mathcal{C}$-modules. 
The corresponding notion and category  
$\mathrm{mod}\text{-}\mathcal{C}$ of {\em right} modules 
are defined similarly using contravariant functors.
The functor ${{}_-}^*$ induces a duality between
$\mathcal{C}\text{-}\mathrm{mod}$ and
$\mathrm{mod}\text{-}\mathcal{C}$. 

Because of our assumptions \eqref{cond0}-\eqref{cond4}, 
indecomposable projective $\mathcal{C}$-modules are of the form
$\mathcal{C}(x,{}_-)$, $x\in\mathcal{C}$, and belong to 
$\mathcal{C}\text{-}\mathrm{mod}$.
Hence $\mathcal{C}\text{-}\mathrm{mod}$ is an abelian length category
(i.e. every object has finite length)
with enough projectives and injectives (because of ${{}_-}^*$). 
Let $\overline{\mathcal{C}}$ denote the full subcategory of 
$\mathcal{C}\text{-}\mathrm{mod}$ consisting of the indecomposable
projectives $\mathcal{C}(x,{}_-)$, $x\in \mathcal{C}$. Then 
$\mathcal{C}\text{-}\mathrm{mod}$ is
equivalent to the category of finite-dimensional 
$\overline{\mathcal{C}}^{\mathrm{op}}$-modules, moreover, 
the category $\overline{\mathcal{C}}^{\mathrm{op}}$ is isomorphic to 
$\mathcal{C}$.

\begin{remark}\label{remreferee}
{\rm  
For an slf-category $\mathcal{C}$ denote by $A_{\mathcal{C}}$
the path algebra of $\overline{\mathcal{C}}^{\mathrm{op}}$. 
Then the category $\mathcal{C}\text{-}\mathrm{mod}$ is equivalent
to the category of finite-dimensional $A_{\mathcal{C}}$-modules.
The algebra $A_{\mathcal{C}}$ can be axiomatically described by
the following properties:
\begin{itemize}
\item $A_{\mathcal{C}}$ is equipped with a system 
$\{e_{\mathtt{i}}:\mathtt{i}\in\cC\}$ of primitive
orthogonal idempotents (the identity morphisms for objects of $\cC$);
\item each $A_{\mathcal{C}}e_{\mathtt{i}}$ and 
$e_{\mathtt{i}}A_{\mathcal{C}}$ is finite dimensional
(this combines \eqref{cond1}--\eqref{cond3});
\item $\displaystyle A_{\mathcal{C}}=
\bigoplus_{\mathtt{i},\mathtt{j}\in\cC}
e_{\mathtt{i}}A_{\mathcal{C}}e_{\mathtt{j}}$.
\end{itemize}

}
\end{remark}

A {\em $\mathcal{C}$-bimodule} is a bifunctor 
$B=B({}_-,{}_-)$ from $\mathcal{C}$ to $\Bbbk\text{-}\mathrm{Mod}$,
contravariant in the first (left) argument and covariant in the second
(right) argument. A typical example of a bimodule is the regular
bimodule $\mathcal{C}=\mathcal{C}({}_-,{}_-)$.

The Nakayama functor $\mathrm{N}:=\mathcal{C}^*\otimes_{\mathcal{C}}{}_-
\cong\mathrm{Hom}_{\mathcal{C}}({}_-,\mathcal{C})^*$ 
is an endofunctor of 
$\mathcal{C}\text{-}\mathrm{mod}$ (for more details on  
tensor  products see \cite[2.2]{MOS}). Consider the endofunctor
$\cL\mathrm{N}$ of $\mathcal{D}^{-}(\mathcal{C})$.
Let $\cP(\mathcal{C})$ denote the full subcategory of 
$\mathcal{D}^{-}(\mathcal{C})$ 
consisting of objects isomorphic to
finite complexes of projective objects  
(the so-called {\em perfect complexes}). Our first easy observation
is the following (compare with \cite[4.3]{MS}):

\begin{proposition}\label{prop1}
Assume that all injective $\mathcal{C}$-modules are of finite 
projective dimension. Then $\cL\mathrm{N}$ is a Serre
functor on $\cP(\mathcal{C})$.
\end{proposition}

\begin{proof}
To start with we claim that $\cL\mathrm{N}$ preserves
$\cP(\mathcal{C})$. Indeed, the functor  $\cL\mathrm{N}$ is
a triangle functor and $\cP(\mathcal{C})$ is generated, as
a triangulated category, by indecomposable projective
$\mathcal{C}$-modules. The functor $\mathrm{N}$ maps
an indecomposable projective $\mathcal{C}$-module to an 
indecomposable injective $\mathcal{C}$-module, and the latter
has finite projective dimension and hence is an object of 
$\cP(\mathcal{C})$.

That $\cL\mathrm{N}$ has the property given by \eqref{eq1} 
is checked by the following standard computation: For any 
$N\in\cP(\mathcal{C})$ and projective 
$P\in \mathcal{C}\text{-}\mathrm{mod}$ we have natural isomorphisms
\begin{displaymath}
\begin{array}{rcl}
\mathrm{Hom}_{\ccP(\mathcal{C})} (N,\cL\mathrm{N}\, P)^* &=& 
\mathrm{Hom}_{\ccP(\mathcal{C})} (N,
\mathrm{Hom}_{\Bbbk}(\mathrm{Hom}_{\mathcal{C}}
(P,\mathcal{C}),\Bbbk))^*\\
\text{(by adjunction)}&=& 
\mathrm{Hom}_{\Bbbk} (\mathrm{Hom}_{\mathcal{C}}
(P,\mathcal{C})\otimes_{\mathcal{C}}
N,\Bbbk)^*\\ \text{(as $({}_-{}^*)^*=\mathrm{Id}$)}
&=& \mathrm{Hom}_{\mathcal{C}}
(P,\mathcal{C})\otimes_{\mathcal{C}}N\\
\text{(by projectivity of $P$)}&=& \mathrm{Hom}_{\mathcal{C}}(P,N).
\end{array}
\end{displaymath}
Using the triangle property for $\cL\mathrm{N}$ this extends to
the whole of $\cP(\mathcal{C})$ (in the second variable) and proves
\eqref{eq1}.
\end{proof}

\subsection{Serre functors and symmetric algebras}\label{s1.3}

Let $\mathcal{C}$ be as in the previous subsection. The category
$\mathcal{C}$ is called {\em symmetric} provided that the
$\mathcal{C}\text{-}\mathcal{C}$-bimodules $\mathcal{C}$
and $\mathcal{C}^*$ are isomorphic. Our second observation 
is the following infinite-dimensional generalization of \cite[Lemma~3.1]{MS}:

\begin{proposition}\label{prop2}
Assume that all injective $\mathcal{C}$-modules are of finite 
projective dimension. Then $\mathcal{C}$ is symmetric if and only if
the Serre functor on $\cP(\mathcal{C})$ is isomorphic to the identity.
\end{proposition}

\begin{proof}
By Proposition~\ref{prop1}, under our assumptions
the functor $\cL\mathrm{N}$ is the 
Serre functor on $\cP(\mathcal{C})$. If $\mathcal{C}$ is symmetric,
then $\mathrm{N}=\mathcal{C}^*\otimes_{\mathcal{C}}{}_-
\cong \mathcal{C}\otimes_{\mathcal{C}}{}_-$, the latter being isomorphic to
the identity functor on $\mathcal{C}\text{-}\mathrm{mod}$.
Hence $\cL\mathrm{N}$ is isomorphic to the identity functor.

Conversely, if $\cL\mathrm{N}$ is isomorphic to the identity 
functor, then $\mathrm{N}$
is isomorphic to the identity functor, when restricted to the
additive subcategory of all projective modules in 
$\mathcal{C}\text{-}\mathrm{mod}$. The latter is certainly true for the
functor $\mathcal{C}\otimes_{\mathcal{C}}{}_-$. As both
$\mathrm{N}$ and $\mathcal{C}\otimes_{\mathcal{C}}{}_-$
are right exact and agree on projective modules, they must be 
isomorphic, which implies that the bimodules 
$\mathcal{C}^*$ and $\mathcal{C}$
are isomorphic.
\end{proof}

\section{Categories with full projective functors}\label{s2}

\subsection{Definitions}\label{s2.1}

Let $\cA$ be an abelian category, $M\in \cA$ and 
$\mathcal{F}:=\{\mathrm{F}_i:i\in I\}$ a full subcategory of the
category  of right exact endofunctors of $\cA$.
We assume that $\mathcal{F}$ is closed (up to isomorphism) under 
direct sums and composition. 
Following \cite{Kh} we say that $(\cA,M,\mathcal{F})$  
is a {\em category with full projective functors}, or, simply, an
{\em fpf-category}, provided that 
\begin{enumerate}[$($i$)$]
\item\label{fpf1} $\mathrm{Id}_{\ccA}\in \mathcal{F}$;
\item\label{fpf2} for every $i\in I$ the object 
$\mathrm{F}_i\, M$ is projective in $\cA$;
\item\label{fpf3} every $N\in \cA$ is a quotient of 
$\mathrm{F}_i\, M$ for some $i\in I$;
\item\label{fpf4} for all $i,j\in I$ the evaluation map
$\mathrm{ev}_M:\mathcal{F}(\mathrm{F}_i,\mathrm{F}_j)\to
\cA(\mathrm{F}_i\, M,\mathrm{F}_j\, M)$ is surjective.
\end{enumerate}
Functors in $\mathcal{F}$ are called {\em projective functors}
and $M$ is called the {\em dominant object}.

\subsection{Basic examples}\label{s2.2}

Let $\mathfrak{g}$ be a semi-simple finite-dimensional complex
Lie algebra with a fixed triangular decomposition 
$\mathfrak{g}=\mathfrak{n}_-\oplus \mathfrak{h}\oplus\mathfrak{n}_+$.
Consider the BGG category $\mathcal{O}$ associated with this
triangular decomposition
(for details on $\mathcal{O}$ we refer the reader to \cite{Hu}). 
Let $\Theta$ be the set of weights of all finite dimensional
$\mathfrak{g}$-modules. For $\lambda\in 
\mathfrak{h}^*$ consider the coset $\overline{\lambda}=\lambda+
\Theta$ and denote by $\mathcal{O}_{\overline{\lambda}}$
the full subcategory of $\mathcal{O}$ containing all modules
$M$ such that $\mathrm{supp}\, M\subset \overline{\lambda}$. 
Fix any dominant regular $\mu\in \overline{\lambda}$ and consider
the corresponding Verma module $\Delta(\mu)$ (with highest weight $\mu$).

For every finite-dimensional $\mathfrak{g}$-module $E$ the functor
$E\otimes_{\mathbb{C}}{}_-$ preserves $\mathcal{O}_{\overline{\lambda}}$.
Let $\mathcal{F}$ denote the family of all such functors.
Then $(\mathcal{O}_{\overline{\lambda}},\Delta(\mu),\mathcal{F})$ is
an fpf-category, see \cite[Proposition~16]{Kh}.

If $\mathfrak{p}$ is a parabolic subalgebra of $\mathfrak{g}$
containing $\mathfrak{h}\oplus\mathfrak{n}_+$, we have the
full subcategory $\mathcal{O}_{\overline{\lambda}}^{\mathfrak{p}}$
of $\mathcal{O}_{\overline{\lambda}}$ consisting of all modules
on which the action of $\mathfrak{p}$ is locally finite. 
Every functor in $\mathcal{F}$ preserves 
$\mathcal{O}_{\overline{\lambda}}^{\mathfrak{p}}$.
Let $\Delta(\mu)^{\mathfrak{p}}$ denote the maximal quotient of 
$\Delta(\mu)$ which lies in $\mathcal{O}_{\overline{\lambda}}^{\mathfrak{p}}$.
Then $(\mathcal{O}_{\overline{\lambda}}^{\mathfrak{p}},
\Delta(\mu)^{\mathfrak{p}},\mathcal{F})$ is
an fpf-category, see \cite[Proposition~22]{Kh}.
For further examples of categories with full projective functors 
see \cite{Kh}.

\subsection{Functors naturally commuting with projective 
functors}\label{s2.3}

Assume that $(\cA,M,\mathcal{F})$ is an fpf-category as 
in Subsection~\ref{s2.1}. An endofunctor 
$\mathrm{G}$ of $\cA$ is said to {\em naturally commute}
with projective functors if for every $i\in I$ there is an isomorphism
$\eta_i:\mathrm{F}_i\circ \mathrm{G}\to \mathrm{G}\circ\mathrm{F}_i$ 
such that for any $i,j\in I$ and any $\alpha\in 
\mathcal{F}(\mathrm{F}_i,\mathrm{F}_j)$
the diagram
\begin{displaymath}
\xymatrix{ 
\mathrm{F}_i\circ\mathrm{G}\ar[rr]^{\alpha_{\mathrm{G}}} 
\ar[d]_{\eta_i} && \mathrm{F}_j\circ\mathrm{G}\ar[d]^{\eta_j}\\
\mathrm{G}\circ\mathrm{F}_i\ar[rr]^{\mathrm{G}(\alpha)} 
&& \mathrm{G}\circ\mathrm{F}_j\\
}
\end{displaymath}
commutes. A functor naturally commuting with projective functors is
determined uniquely (up to isomorphism) by its image on
the dominant object $M$, see \cite{Kh}. For examples of 
functors naturally commuting with 
projective functors (in the situations described in 
Subsection~\ref{s2.2}) we again refer the reader to \cite{Kh}.
A natural question to ask is under what assumptions the 
Nakayama functor naturally commutes with projective functors.

\section{Serre functors for Lie algebras}\label{s3}

\subsection{Harish-Chandra bimodules and category $\mathcal{O}$}\label{s3.1}

From now on $\Bbbk=\mathbb{C}$ and 
we abbreviate $\otimes_{\mathbb{C}}$ by $\otimes$.
Consider the setup of Subsection~\ref{s2.2} and denote by $\mathcal{H}$
the category of {\em Harish-Chandra bimodules} for $\mathfrak{g}$. This 
category consists of all finitely generated $\mathfrak{g}$-bimodules 
on which the adjoint action of $\mathfrak{g}$ is locally finite (see 
\cite{BG} or \cite[Kapitel~6]{Ja}). 

Let $U=U(\mathfrak{g})$ denote the universal enveloping algebra 
of $\mathfrak{g}$. For a primitive ideal $I$ of $U$
let $\mathcal{H}_{I}^1$ denote the full subcategory of $\mathcal{H}$
which consists of all bimodules annihilated by $I$ from the right.
For $\lambda\in\mathfrak{h}^*$ let $I_{\lambda}$ denote the
annihilator of $\Delta(\lambda)$. If $\lambda$ is dominant and
$I$ is a primitive ideal containing $I_{\lambda}$, we denote by
$\Delta^I(\lambda)$ the quotient $\Delta(\lambda)/I\Delta(\lambda)$.
We denote by $\mathcal{O}_I$ the full subcategory of
$\mathcal{O}_{\overline{\lambda}}$ consisting of all modules which are
isomorphic to quotients of modules of the form 
$E\otimes \Delta^I(\lambda)$, where $E$ is a finite-dimensional
$\mathfrak{g}$-module (the categories $\mathcal{O}_0^{\hat{\mathbf{R}}}$ 
from \cite[4.3]{MS2} are direct summands of $\mathcal{O}_I$).
Let ${}_-{}^{\star}$ be the usual duality on $\mathcal{O}$. The duality
${}_-{}^{\star}$ restricts to $\mathcal{O}_I$.

For two $\mathfrak{g}$-modules $M$ and $N$, let $\mathcal{L}(M,N)$ denote the
$\mathfrak{g}$-bimodule of all linear maps from $M$ to $N$ which are
locally finite with respect to the adjoint action of $\mathfrak{g}$
(see \cite[Kapitel~6]{Ja}). For $\lambda\in \mathfrak{h}^*$ dominant
and regular we have an equivalence of categories as
follows (see \cite[Theorem~5.9]{BG} for the category $\mathcal{O}$
and \cite[Theorem~3.1]{MiSo} for the general statement):
\begin{equation}\label{eq2}
\xymatrix{
\mathcal{O}_{I}
\ar@/^7pt/[rrrr]^{\mathcal{L}(\Delta^I(\lambda),{}_-)}&&&&
\mathcal{H}_{I}^1
\ar@/^7pt/[llll]^{{}_-\otimes_{U}\Delta^I(\lambda)}
}
\end{equation}
Fix one representative in every isomorphism class of indecomposable 
projective objects in $\mathcal{O}_{I}$ and denote by $\mathcal{C}_I$ 
the full subcategory of $\mathcal{O}_{I}$ generated by these fixed 
objects. Then $\mathcal{C}_I$ is an slf-category and 
$\mathcal{O}_{I}$ is equivalent to 
$\mathcal{C}_I\text{-}\mathrm{mod}$ (note that 
$\mathcal{C}_I\cong \mathcal{C}_I^{\mathrm{op}}$ because of $\star$).
The canonical equivalence between $\mathcal{O}_{I}$ and
$\mathcal{C}_I^{\mathrm{op}}\text{-}\mathrm{mod}$
sends $M\in \mathcal{O}_{I}$ to the $\mathcal{C}^{\mathrm{op}}_I$-module 
$\mathrm{Hom}_{\mathfrak{g}}({}_-,M)$ (which is a functor
from $\mathcal{C}^{\mathrm{op}}_I$ to $\mathbb{C}\text{-}\mathrm{mod}$) 
and is defined on morphisms in the natural way.

Let $\sigma$ denote the Chevalley anti-involution on $\mathfrak{g}$.
For a $\mathfrak{g}$-bimodule $X$ we denote by $X^s$ the bimodule
defined as follows: $X^s=X$ as a vector space and for $a,b\in \mathfrak{g}$,
$x\in X$, we have $a\cdot_s x\cdot_s b:=\sigma(b)x\sigma(a)$. We also
denote by $X^*$ the dual bimodule of $X$. If $X$ is a Harish-Chandra
bimodule, we denote by $X^{\circledast}$ the subbimodule of 
$X^*$ consisting of all elements on which the adjoint action of 
$\mathfrak{g}$ is locally finite.

\subsection{Serre functor for category $\mathcal{O}$}\label{s3.2}

The Serre functor for the regular block of the category $\mathcal{O}$
was described in \cite{BBM} and \cite{MS}. The latter also contains
a description of the Serre functor for the regular block of 
$\mathcal{O}^{\mathfrak{p}}$.

\begin{theorem}\label{thm4}
Let $\lambda$ be dominant and regular and $I$ be a primitive 
ideal containing $I_{\lambda}$.  Set $\Delta:=\Delta^I(\lambda)$.
Then for any $P,N\in \mathcal{O}_I$ 
with $P$ projective there is an isomorphism
\begin{displaymath}
\mathrm{Hom}_{\mathfrak{g}}\big(N,
\mathcal{L}(P,\Delta)^{\circledast}\otimes_{U}
\Delta\big)\cong \mathrm{Hom}_{\mathfrak{g}}(P,N)^*,
\end{displaymath}
natural in both $P$ and $N$.
\end{theorem}

\begin{proof}
In the proof we will need the following statement:

\begin{proposition}\label{prop55}
For any $P,N\in \mathcal{O}_I$ 
with $P$ projective there is an isomorphism
\begin{equation}\label{eq56}
\mathrm{Hom}_{\mathfrak{g}\text{-}\mathfrak{g}}
(\mathcal{L}(\Delta,P),\mathcal{L}(\Delta,N))\cong
\mathcal{L}(P,\Delta)
\otimes_{U\text{-}U}\mathcal{L}(\Delta,N)
\end{equation}
natural in both $P$ and $N$ (here $\otimes_{U\text{-}U}$ denotes
the tensor product over $U\otimes U^{\mathrm{op}}$).
\end{proposition}

\begin{proof}
We start with the following observation:

\begin{lemma}\label{lem5}
In the case $P=\Delta$ both sides of \eqref{eq56} are finite 
dimensional vector spaces of the same dimension.
\end{lemma}

\begin{proof}
If $P=\Delta$, then $\mathcal{L}(\Delta,\Delta)\cong U/I$, 
see \cite[6.9]{Ja}, and the
left hand side of \eqref{eq56} (which is clearly finite dimensional) 
is isomorphic to  $\mathcal{L}(\Delta,N)^{\mathfrak{g}}$, 
the set of $\mathfrak{g}$-invariants of $\mathcal{L}(\Delta,N)$, see \cite[Lemma~2.2]{BG}.  Further, we have the inclusion
\begin{displaymath}
\begin{array}{rcl}
U/I
\otimes_{U\text{-}U}\mathcal{L}(\Delta,N)&\subset&
\mathrm{Hom}_{\mathbb{C}}\big(U/I
\otimes_{U\text{-}U}\mathcal{L}(\Delta,N),\mathbb{C}\big)^*\\
\text{(by adjunction)}&\cong&
\mathrm{Hom}_{\mathfrak{g}\text{-}\mathfrak{g}}
\big(U/I,\mathcal{L}(\Delta,N)^*\big)^*\\
&\cong&
\mathrm{Hom}_{\mathfrak{g}\text{-}\mathfrak{g}}
\big(U/I,\mathcal{L}(\Delta,N)^{\circledast}\big)^*\\
\text{(by the above)}&\cong&
\big((\mathcal{L}(\Delta,N)^{\circledast})^{\mathfrak{g}}\big)^*\\
&\cong&\mathcal{L}(\Delta,N)^{\mathfrak{g}},
\end{array}  
\end{displaymath}
where the third step is justified by the facts that 
the image of any homomorphism from $U/I$ to $\mathcal{L}(\Delta,N)^*$ 
belongs  to $\mathcal{L}(\Delta,N)^{\circledast}$, and the fifth step 
is justified by the fact that 
$(\mathcal{L}(\Delta,N)^{\circledast})^{\mathfrak{g}}\cong 
(\mathcal{L}(\Delta,N)^{\mathfrak{g}})^*$
as the canonical non-degenerate pairing 
$\mathcal{L}(\Delta,N)\times 
\mathcal{L}(\Delta,N)^{\circledast}\to\mathbb{C}$ restricts to a 
non-degenerate pairing 
$\mathcal{L}(\Delta,N)^{\mathfrak{g}}\times 
(\mathcal{L}(\Delta,N)^{\circledast})^{\mathfrak{g}}\to \mathbb{C}$. 
Since the vector space $\mathcal{L}(\Delta,N)^{\mathfrak{g}}$ is 
finite dimensional, the original inclusion 
\begin{displaymath}
U/I
\otimes_{U\text{-}U}\mathcal{L}(\Delta,N)\subset
\mathrm{Hom}_{\mathbb{C}}\big(U/I
\otimes_{U\text{-}U}\mathcal{L}(\Delta,N),\mathbb{C}\big)^* 
\end{displaymath}
is, in fact, an isomorphism. The claim follows.
\end{proof} 

Write $\mathcal{L}(\Delta,N)\cong \mathcal{L}(\Delta,N)^{\mathfrak{g}}\oplus
X$, where $X$ is stable with respect to the adjoint action of 
$\mathfrak{g}$. In the case $P=\Delta$ we have 
$\mathcal{L}(\Delta,\Delta)=U/I$ and $U/I\otimes_{U-U}\mathcal{L}(\Delta,N)$
is the $0$-degree Hochschild homology of $\mathcal{L}(\Delta,N)$, which
equals 
\begin{equation}\label{eq57}
U/I\otimes_{U-U}\mathcal{L}(\Delta,N)\cong 
\mathcal{L}(\Delta,N)/\mathcal{K}\mathcal{L}(\Delta,N),
\end{equation}
where
\begin{displaymath}
\mathcal{K}\mathcal{L}(\Delta,N):=
\{uv-vu:u\in U/I, v\in \mathcal{L}(\Delta,N)\}.
\end{displaymath}
As an adjoint module, $X$ is a direct sum of finite-dimensional simple
modules, none of which is isomorphic to the trivial module. Therefore
$X\subset \mathcal{K}\mathcal{L}(\Delta,N)$. From
Lemma~\ref{lem5} it follows that $X=\mathcal{K}\mathcal{L}(\Delta,N)$
comparing the dimensions. Composing \eqref{eq57} with the 
projection of  $\mathcal{L}(\Delta,N)$ onto 
$\mathcal{L}(\Delta,N)^{\mathfrak{g}}$ along $X$ we obtain the following 
natural isomorphism:
\begin{equation}\label{eq58}
U/I\otimes_{U-U}\mathcal{L}(\Delta,N)\cong 
\mathcal{L}(\Delta,N)/\mathcal{K}\mathcal{L}(\Delta,N)\cong
\mathcal{L}(\Delta,N)^{\mathfrak{g}}.
\end{equation}

Consider now the case where $P=E\otimes \Delta$ for some finite-dimensional
$\mathfrak{g}$-module $E$. By \cite[6.2]{Ja} we have canonical isomorphisms
\begin{equation}\label{eq59}
\mathcal{L}(E\otimes \Delta,\Delta)\cong
\mathcal{L}(\Delta,\Delta)\otimes E^*,\quad\quad
\mathcal{L}(\Delta,E\otimes \Delta)\cong
E\otimes \mathcal{L}(\Delta,\Delta).
\end{equation}

Denote by $\check{E}$ the right 
$\mathfrak{g}$-module defined as follows: $\check{E}=E$ as a vector 
space and $v\,\check{\cdot}\,g:=-gv$ for $v\in E$ and 
$g\in \mathfrak{g}$ (and similarly for going from right to left modules).
Then the usual adjunction (see \cite[2.1(d)]{BG}) gives,  
by restriction, the isomorphism
\begin{equation}\label{eq6}
\mathrm{Hom}_{\mathfrak{g}\text{-}\mathfrak{g}}
(E\otimes U/I,X)\cong
\mathrm{Hom}_{\mathfrak{g}\text{-}\mathfrak{g}}
(U/I,\check{E}^{*}\otimes X).
\end{equation}

A straightforward computation also shows that 
\begin{equation}\label{eq5}
(U/I\otimes E^{*})
\otimes_{U\text{-}U} X\cong
U/I\otimes_{U\text{-}U} 
(\check{E}^{*}\otimes X)
\end{equation}
via $u\otimes v\otimes x\mapsto u\otimes v\otimes x$.

Combining \eqref{eq58}, \eqref{eq5} and \eqref{eq6} we
obtain a required isomorphism in the case $P=E\otimes \Delta$.
Now the claim follows from additivity of all functors and
naturality of all constructions.
\end{proof} 
 
The theorem is now proved via the following chain 
of natural isomorphisms:
\begin{displaymath}
\begin{array}{rcl} 
\mathrm{Hom}_{\mathfrak{g}}(P,N)&\overset{\eqref{eq2}}{\cong}&
\mathrm{Hom}_{\mathfrak{g}\text{-}\mathfrak{g}}
(\mathcal{L}(\Delta,P),\mathcal{L}(\Delta,N))\\
\text{(by Proposition~\ref{prop55})}&\cong&
\mathcal{L}(P,\Delta)
\otimes_{U\text{-}U}\mathcal{L}(\Delta,N)\\
\text{(taking the double dual)}&\cong&
\mathrm{Hom}_{\mathbb{C}}\big(\mathcal{L}(P,\Delta)
\otimes_{U\text{-}U}
\mathcal{L}(\Delta,N),\mathbb{C}\big)^*\\
\text{(by adjunction)}&\cong&
\mathrm{Hom}_{\mathfrak{g}\text{-}\mathfrak{g}}\big(\mathcal{L}(\Delta,N),
\mathcal{L}(P,\Delta)^*\big)^*\\
&\cong&\mathrm{Hom}_{\mathfrak{g}\text{-}\mathfrak{g}}
\big(\mathcal{L}(\Delta,N),
\mathcal{L}(P,\Delta)^{\circledast}\big)^*\\
\text{(by \eqref{eq2})}&{\cong}&
\mathrm{Hom}_{\mathfrak{g}}\big(N,
\mathcal{L}(P,\Delta)^{\circledast}\otimes_{U}
\Delta\big)^*
\end{array}
\end{displaymath}
Here taking the double dual is justified by the fact that the vector
space in question is finite dimensional (see Lemma~\ref{lem5}), 
and the penultimate step
is justified by the fact that the image of every morphism 
from $\mathcal{L}(\Delta,N)$ to $\mathcal{L}(P,\Delta)^*$
belongs to $\mathcal{L}(P,\Delta)^{\circledast}$.
\end{proof}

\begin{corollary}\label{cor7}
Assume that $\lambda$, $I$ and $\Delta$ are as above. 
\begin{enumerate}[$($a$)$]
\item\label{cor7.1} The functor
$\mathcal{L}({}_-,\Delta)^{\circledast}\otimes_{U}\Delta$
is isomorphic to the Nakayama functor on $\mathcal{O}_I$.
\item\label{cor7.2} If we additionally assume that all injective 
modules in $\mathcal{O}_I$ have finite projective dimension, then
the left derived of the functor
$\mathcal{L}({}_-,\Delta)^{\circledast}\otimes_{U}\Delta$
is a Serre functor on $\cP(\mathcal{C}_I)$.
\end{enumerate}
\end{corollary}

We would like to point out that the hypothesis of 
Corollary~\ref{cor7}\eqref{cor7.2} 
is satisfied if $\mathcal{O}_I$ is a direct summand of
the usual or parabolic category $\mathcal{O}$.
Note that $\mathcal{O}_I$ is a direct summand of $\mathcal{O}$
if $I$ is a minimal primitive ideal. Furthermore, $\mathcal{O}_I$ 
is a direct summand of some parabolic $\mathcal{O}$ if $I$ is 
the annihilator of some dominant parabolic Verma module.

\subsection{Alternative descriptions and applications}\label{s3.3}

\begin{corollary}\label{cor8}
Under the assumptions of Corollary~\ref{cor7}\eqref{cor7.1} 
the Nakayama functor on $\mathcal{O}_I$ naturally commutes with 
projective functors.
\end{corollary}

\begin{proof}
By Corollary~\ref{cor7}\eqref{cor7.1}, the Nakayama functor
on $\mathcal{O}_I$ is isomorphic to the functor
$\mathcal{L}({}_-,\Delta)^{\circledast}\otimes_{U}\Delta$.
Applying ${}_-{}^{\circledast}$ to the left isomorphism in
\eqref{eq59} we get
that $\mathcal{L}({}_-,\Delta)^{\circledast}\otimes_{U}\Delta$
commutes with projective functors and naturality follow from
the definition of projective functors.
\end{proof}

Under the assumptions of Corollary~\ref{cor7}\eqref{cor7.1} 
consider the endofunctor $\mathrm{C}$ of $\mathcal{O}_I$
of {\em partial coapproximation} with respect to projective-injective
modules, see \cite[2.5]{KhMa}. This functor is the unique 
(up to isomorphism) right
exact functor which sends a projective module $P$ to the submodule of
$P$ generated by images of all possible morphisms from projective-injective
modules to $P$ and acts on morphisms via restriction.

The module $\Delta=\Delta^I(\lambda)$ has, 
by \cite[Corollary~3]{Ma}, simple socle which we denote by 
$K=K^I(\lambda)$ (it is the only simple subquotient of
$\Delta$ of maximal Gelfand-Kirillov (GK) dimension). 
We have a canonical
embedding $\varphi:U/\mathrm{Ann}_{U}K
\hookrightarrow \mathcal{L}(K,K)$. The question of surjectivity of 
$\varphi$ is known as {\em Kostant's problem} for $K$ (see \cite{KM}
and references therein). It is known (see \cite[4.1]{KM}) that 
Kostant's problem has a positive answer in the case when 
$\mathcal{O}_{I}$ is the usual or parabolic category $\mathcal{O}$
(and in many other cases as well).

\begin{corollary}\label{cor9}
Assume that $\lambda$, $I$ and $\Delta$ are as above
and that $\varphi$ is surjective. Then the 
Nakayama functor on $\mathcal{O}_I$ is isomorphic to $\mathrm{C}^2$.
\end{corollary}

\begin{proof}
The fact that functor $\mathrm{C}$ (and hence also $\mathrm{C}^2$) 
naturally commutes  with projective functors repeats verbatim the proof of 
\cite[Claim~2, Page~153]{MS}. The Nakayama functor
naturally commutes with projective functors by Corollary~\ref{cor8}.
Therefore to complete the proof we only need to show that 
$\mathrm{C}^2$ maps $\Delta$ to the corresponding injective module.

First we claim that $\mathrm{C}\Delta\cong K$, where $K$ is the
simple socle of $\Delta$ (as in Subsection~\ref{s3.2}).
Indeed, $\Delta$ is projective and hence $\mathrm{C}\Delta$ coincides 
with the trace of projective-injective modules in $\Delta$. 
By \cite[3.2]{Ma}, this trace coincides with $K$.

By \cite[Corollary~12]{Ma}, the surjectivity of $\varphi$ is equivalent
to the existence of a two step resolution $\Delta\hookrightarrow X_0\to X_1$
with $X_0$ and $X_1$ projective-injective. Applying $\star$ we obtain
\begin{equation}\label{eq7}
X_1^{\star}\to X_0^{\star}\tto \Delta^{\star}
\end{equation}
where both $X_0^{\star}$ and $X_1^{\star}$ are again projective-injective.
As the kernel of the natural projection 
$\Delta^{\star}\tto K^{\star}\cong K$ is killed by $\mathrm{C}$, it 
follows that  $\mathrm{C} K\cong \mathrm{C}\Delta^{\star}$
and from \eqref{eq7} we obtain $\mathrm{C}\Delta^{\star}\cong\Delta^{\star}$,
which is exactly what we needed.
\end{proof}

\begin{remark}\label{rem10}
{\rm  
If $\mathcal{O}_I$ is a direct summand of the usual or parabolic 
category $\mathcal{O}$ one can give yet another description for
the Nakayama and Serre functors in terms of Arkhipov's twisting
functors, see \cite{MS} for details.
}
\end{remark}

Fix one representative in every isomorphism class of indecomposable
projective-injective objects in $\mathcal{O}_I$ and denote by 
$\mathcal{P}_I$ the full subcategory of $\mathcal{O}_I$ generated
by these fixed objects. The category $\mathcal{P}_I$ is an
slf-category.

\begin{corollary}\label{cor11}
Under the assumptions of Corollary~\ref{cor9}
the category $\mathcal{P}_I$ is symmetric.
\end{corollary}

\begin{proof}
By our construction of $\mathcal{P}_I$, every projective 
$\mathcal{P}_I$-module is also injective. Thus 
the Nakayama functor preserves $\cP(\mathcal{P}_I)$ and hence
gives rise to a Serre functor on this category.
From Corollary~\ref{cor9} we obtain that the Serre functor 
on $\cP(\mathcal{P}_I)$ is the left derived of $\mathrm{C}^2$. 
The functor $\mathrm{C}$ obviously induces the identity functor when 
restricted to projective-injective modules.
Now the claim follows from Proposition~\ref{prop2}.
\end{proof}

In the case when $I$ is the annihilator of some parabolic Verma module,
Corollary~\ref{cor11} implies (and reproves) \cite[Theorem~4.6]{MS}.
If $\mathfrak{g}$ is of type $A$, then for an arbitrary $I$ the claim 
of  Corollary~\ref{cor11} can be deduced combining
\cite[Theorem~4.6]{MS} and \cite[Theorem~18]{MS2}. In all
other cases Corollary~\ref{cor11} seems to be new.

\section{Serre functors for Lie superalgebras}\label{s4}

\subsection{Generalities on Lie superalgebras}\label{s4.1}

In this section we denote by $\mathfrak{g}$ one of the following
Lie superalgebras: $\mathfrak{gl}(m,n)$, $\mathfrak{sl}(m,n)$,
$\mathfrak{osp}(m,n)$, $\mathfrak{psl}(n,n)$, $\mathfrak{q}(n)$,
$\mathfrak{pq}(n)$, $\mathfrak{sq}(n)$, $\mathfrak{psq}(n)$. For all
these superalgebras the Lie algebra $\mathfrak{g}_0$ is reductive. 
The first four superalgebras admit an even $\mathfrak{g}$-invariant bilinear 
form which is non-degenerate on $[\mathfrak{g},\mathfrak{g}]$
(these superalgebras are called {\em basic}). The last four superalgebras
are called {\em queer} Lie superalgebras, or Lie superalgebras 
of type $\mathfrak{q}$.
In the following we denote by $\mathfrak{g}\text{-}\mathrm{sMod}$ 
the category of all $\mathfrak{g}$-supermodules, where morphisms are
homogeneous of degree zero (and by a module over a superalgebra we
always mean a supermodule). The category $\mathfrak{g}\text{-}\mathrm{sMod}$
is abelian. Denote by $\Pi$ the parity change autoequivalence of 
$\mathfrak{g}\text{-}\mathrm{sMod}$ and by $U=U(\mathfrak{g})$
the universal enveloping algebra of $\mathfrak{g}$.

Fix a triangular decomposition $\mathfrak{g}=\mathfrak{n}^-\oplus
\mathfrak{h}\oplus\mathfrak{n}^+$ of $\mathfrak{g}$, where 
$\mathfrak{h}$ is a Cartan subalgebra, and let $\mathcal{O}$ denote
the full subcategory of $\mathfrak{g}\text{-}\mathrm{sMod}$
consisting of finitely generated modules $M$ such that the
action of $\mathfrak{h}_0$ on $M$ is diagonalizable and the
action of $U(\mathfrak{n}^+)$ on $M$ is locally finite. Note that
$\mathfrak{h}=\mathfrak{h}_0$ if and only if  $\mathfrak{g}$ is basic.

If $\mathfrak{p}\subset \mathfrak{g}$ is a parabolic subsuperalgebra
containing $\mathfrak{h}\oplus\mathfrak{n}^+$, we denote by 
$\mathcal{O}^{\mathfrak{p}}$ the full subcategory of $\mathcal{O}$
consisting of all modules $M$ on which the action of
$U(\mathfrak{p})$ is locally finite. In particular, the category
$\mathcal{O}^{\mathfrak{g}}$ coincides with the category of 
finite dimensional $\mathfrak{h}_0$-diagonalizable 
$\mathfrak{g}$-modules. Abusing language, in the following by 
finite dimensional $\mathfrak{g}$-modules we will mean
finite dimensional $\mathfrak{h}_0$-diagonalizable $\mathfrak{g}$-modules
(the two categories coincide if $\mathfrak{g}_0$ is
a semi-simple Lie algebra). Let ${}_-{}^{\star}$ denote the usual duality 
on  $\mathcal{O}$ (which is simple preserving for basic
$\mathfrak{g}$ and simple preserving up to $\Pi$  if
$\mathfrak{g}$ is of type $\mathfrak{q}$). 
This duality restricts to $\mathcal{O}^{\mathfrak{p}}$.

We have the usual induction and restriction functors $\mathrm{Ind}$
and $\mathrm{Res}$ between $\mathfrak{g}\text{-}\mathrm{sMod}$ and
$\mathfrak{g}_0\text{-}\mathrm{sMod}$. 
They are adjoint in the 
usual way $(\mathrm{Ind},\mathrm{Res})$. We also have the adjointness
$(\mathrm{Res},\Pi^{\dim\mathfrak{g}_1}\circ\mathrm{Ind})$, see
\cite[3.2.4]{Go1} and \cite[Proposition~2.2]{Fr}. Note that 
$\dim\mathfrak{g}_1$ is even dimensional if $\mathfrak{g}$ is basic
and also for $\mathfrak{q}(n)$ and $\mathfrak{pq}(n)$ when $n$ is even,
and for $\mathfrak{sq}(n)$ and $\mathfrak{psq}(n)$ when $n$ is odd. 

Since $U(\mathfrak{g})$ is a finite extension of $U(\mathfrak{g}_0)$,
we define the category $\mathcal{H}$ of {\em Harish-Chandra
$\mathfrak{g}$-bimodules} as the full subcategory of the category
of $\mathfrak{g}$-bimodules which become Harish-Chandra
$\mathfrak{g}_0$-bimodules after restriction. For any graded ideal
$I$ of $U(\mathfrak{g})$ we denote by $\mathcal{H}^1_I$ the full
subcategory of $\mathcal{H}$ which consists of all bimodules
annihilated by $I$ from the right.
If $M$ and $N$ are $\mathfrak{g}$-modules, we denote by
$\mathcal{L}(M,N)$ the $\mathfrak{g}_0$-bimodule 
$\mathcal{L}(\mathrm{Res}\,M,\mathrm{Res}\,N)$ which is also a
$\mathfrak{g}$-bimodule as $U(\mathfrak{g})$ is a finite 
extension of $U(\mathfrak{g}_0)$.

\subsection{Structure of the category $\mathcal{O}$}\label{s4.2}

Consider the category $\mathfrak{h}\text{-}\mathrm{dmod}$ of 
finite dimensional $\mathfrak{h}$-modules on which the action of
$\mathfrak{h}_0$ is diagonalizable. In case $\mathfrak{h}=\mathfrak{h}_0$
(i.e. $\mathfrak{g}$ is basic)
this category is semi-simple and its simple objects are naturally
parameterized by pairs $(\lambda,\varepsilon)$, where 
$\lambda\in \mathfrak{h}_0^*$ (which describes the 
action of $\mathfrak{h}_0$) and $\varepsilon\in\{0,1\}$
(which determines the parity of the module). If $\mathfrak{g}$ is of type
$\mathfrak{q}$, then the category $\mathfrak{h}\text{-}\mathrm{dmod}$
is not semi-simple, it has enough projectives (see \cite[Section~3]{Br1}) 
and its simple objects are naturally parameterized by pairs 
$(\lambda,\varepsilon)$, where $\lambda\in \mathfrak{h}_0^*$ (which 
describes the action of $\mathfrak{h}_0$) and $\varepsilon$ is either in 
$\{+\}$ or in $\{+,-\}$ depending on $\lambda$ as prescribed by the theory
of Clifford algebras (see \cite[Appendix]{Go2}). If we have two different
simple modules for some $\lambda$, they differ by a parity change.

For a simple $V\in \mathfrak{h}\text{-}\mathrm{dmod}$ set 
$\mathfrak{n}^+V=0$ and define the corresponding {\em Verma} 
module as
\begin{displaymath}
{\Delta}(V):=U(\mathfrak{g})\otimes_{U(\mathfrak{h}\oplus
\mathfrak{n}^+)} V.
\end{displaymath}
Every object in $\mathcal{O}$ has finite length
(already as a $\mathfrak{g}_0$-module as $U(\mathfrak{g})$ is a 
finite  extension of $U(\mathfrak{g}_0)$). To distinguish
Lie algebras from superalgebras, we denote by $\tilde{\mathcal{O}}$
the category $\mathcal{O}$ for $\mathfrak{g}_0$ (defined with respect to
the triangular decomposition of $\mathfrak{g}_0$ obtained from the 
above triangular decomposition of $\mathfrak{g}$ by restriction).

Category $\mathcal{O}$ has enough projectives (which may be obtained as
direct summands of modules induced from projectives in  
$\tilde{\mathcal{O}}$), every projective has a filtration by
Verma modules (see \cite{Br1}).
As induction is exact and adjoint (from both sides) to exact functors,
it sends injectives to injectives and projectives to projectives.
In particular, it follows that all injective modules in 
$\mathcal{O}$ have finite projective dimension (since this is true
for $\tilde{\mathcal{O}}$).

For any parabolic $\mathfrak{p}$ (as in Subsection~\ref{s4.1}), 
the usual projection from
$\mathcal{O}$ onto $\mathcal{O}^{\mathfrak{p}}$ obviously commutes
with both induction and restriction. Hence, using \cite{RC}
and \cite{Sc}, we similarly get that 
injective modules in  $\mathcal{O}^{\mathfrak{p}}$ have 
finite projective dimension. We denote by 
${\Delta}^{\mathfrak{p}}(V)$ the image of $\Delta(V)$ 
in $\mathcal{O}^{\mathfrak{p}}$.

\subsection{Equivalence to Harish-Chandra bimodules}\label{s4.3}

Let $T$ denote the generator of the anticenter of $U(\mathfrak{g})$,
see \cite{Go1}. An element $\lambda\in \mathfrak{h}_0^*$ is called
{\em strongly typical} if $T$ does not annihilate the Verma 
module ${\Delta}(V)$, where $V$ is a simple 
$\mathfrak{h}$-module on which $\mathfrak{h}_0$ acts via $\lambda$.
Note that $\lambda$ strongly typical implies that $V$ is projective
in $\mathfrak{h}\text{-}\mathrm{dmod}$.

Let $\Theta$ denote the set of weights of all simple finite 
dimensional $\mathfrak{g}$-modules. Fix a strongly typical 
$\lambda\in \mathfrak{h}_0^*$ which is regular and dominant  
with respect to the dot-action of $W$ and set $\overline{\lambda}=
\lambda+\Theta$. Call such $\lambda$ {\em generic} provided
that the $\mathfrak{g}_0$-module $\mathrm{Res}\,\Delta(V)$ is a
direct sum of Verma modules and non-isomorphic direct summands of
$\mathrm{Res}\,\Delta(V)$ correspond to different central characters.
Denote by $\mathcal{O}_{\overline{\lambda}}$ the full
subcategory of $\mathcal{O}$ consisting of all $M$ such that
the $\mathfrak{g}_0$-support of $M$ belongs to $\overline{\lambda}$.
Define $\mathcal{O}_{\overline{\lambda}}^{\mathfrak{p}}$ correspondingly.
For $V$ as above the indecomposable direct summand of 
$\mathcal{O}$ containing ${\Delta}(V)$ is  equivalent
to a direct summand of $\tilde{\mathcal{O}}$,
see \cite{Go3,FM}.

As in the Lie algebra situation, we have {\em projective endofunctors}
on $\mathcal{O}$ given as direct summands of the functors of the form
$E\otimes{}_-$, where $E$ is a finite dimensional
$\mathfrak{g}$-module. The functor $E\otimes{}_-$ is both
left and right adjoint to $\check{E}^*\otimes{}_-$, 
in particular, it is exact and sends projectives to projectives.
It is easy to see that every indecomposable projective in 
$\mathcal{O}_{\overline{\lambda}}^{\mathfrak{p}}$ is a direct summand
of $E\otimes{\Delta}(V)$ for some finite dimensional $E$, 
in particular,
$\mathcal{O}_{\overline{\lambda}}^{\mathfrak{p}}$ coincides with the
category $\mathrm{coker}\langle E\otimes{\Delta}(V)
\rangle$ which consists of all modules $M$ having a two step resolution
\begin{displaymath}
X_1\to X_0\to M\to 0, 
\end{displaymath}
where $X_0,X_1$ belong to the additive closure of modules of the form
$E\otimes{\Delta}(V)$ for finite dimensional $E$. Denote by $\mathcal{F}$ 
the full subcategory of the category of right exact endofunctors of 
$\mathcal{O}$ whose  objects are projective functors. 
Now we can formulate our first main result for superalgebras.

\begin{theorem}\label{thm21}
Let $\lambda\in \mathfrak{h}_0^*$ be  as above, $V$ a simple
$\mathfrak{h}_0$-module of weight $\lambda$ and
$I=\mathrm{Ann}_{U}\Delta^{\mathfrak{p}}(V)$.
\begin{enumerate}[$($a$)$]
\item\label{thm21.3} The triple 
$(\mathcal{O}_{\overline{\lambda}}^{\mathfrak{p}},
\Delta^{\mathfrak{p}}(V),\mathcal{F})$ is an fpf-category. 
\item\label{thm21.1} Assume that $\Pi V\not\cong
V$ and that $\lambda$ is generic. Then the natural injective map
$U/I\hookrightarrow \mathcal{L}(\Delta^{\mathfrak{p}}(V),
\Delta^{\mathfrak{p}}(V))$ is surjective.
\item\label{thm21.2} Assume that $\Pi V\not\cong V$ and that 
$\lambda$ is generic. Then we have the following mutually inverse 
equivalences of categories:
\begin{equation}\label{eq21}
\xymatrix{
\mathcal{O}_{\overline{\lambda}}^{\mathfrak{p}}
\ar@/^7pt/[rrrr]^{\mathcal{L}(\Delta^{\mathfrak{p}}(V),{}_-)}&&&&
\mathcal{H}_{I}^1
\ar@/^7pt/[llll]^{{}_-\otimes_{U}\Delta^{\mathfrak{p}}(V)}
}
\end{equation}
\end{enumerate}
\end{theorem}

Note that for our choice of $\lambda$ the module 
$\Delta^{\mathfrak{p}}(V)$ is obviously projective in 
$\mathcal{O}^{\mathfrak{p}}$. The condition $\Pi V\not\cong V$
is equivalent to the condition that $\dim\mathfrak{g}_1$ is even.

\subsection{Proof of Theorem~\ref{thm21}\eqref{thm21.3}}\label{s4.4}

We only need to check condition \eqref{fpf4}. Assume first that
$\mathfrak{p}=\mathfrak{h}\oplus \mathfrak{n}^+$. For any finite
dimensional $\mathfrak{g}$-module $E$ we have the following
commutative diagram:
\begin{equation}\label{eq35}
\xymatrix{
\mathcal{O}\ar[d]_{\mathrm{Res}}\ar[rr]^{E\otimes{}_-} 
&&\mathcal{O}\ar[d]^{\mathrm{Res}}\\
\tilde{\mathcal{O}}\ar[rr]^{\mathrm{Res}E\otimes{}_-} &&\tilde{\mathcal{O}}
}
\end{equation}
Denote by $\mathcal{O}_V$ the block of $\mathcal{O}$ containing $\Delta(V)$.
For any $\mu\in \mathfrak{h}_0^*$ denote 
by $\tilde{\mathcal{O}}_{\mu}$ the
block of $\tilde{\mathcal{O}}$ containing $\Delta(\mu)$.
By \cite{Go3} and \cite{FM}, to our $\lambda$ there corresponds
a dominant regular $\lambda'\in \mathfrak{h}_0^*$ such that 
the appropriate direct summand of \eqref{eq35} has the form
\begin{equation}\label{eq36}
\xymatrix{
\mathcal{O}_V\ar[d]_{\mathrm{G}}\ar[rr]^{\mathrm{F}} 
&&\mathcal{O}_V\ar[d]^{\mathrm{G}}\\
\tilde{\mathcal{O}}_{\lambda'}\ar[rr]^{\mathrm{F}'} 
&&\tilde{\mathcal{O}}_{\lambda'},
}
\end{equation}
where $\mathrm{G}$ is an equivalence. In particular, for any projective
functor $\mathrm{F}:\mathcal{O}_V\to \mathcal{O}_V$ there is a
projective functor $\mathrm{F}':\tilde{\mathcal{O}}_{\lambda'}\to
\tilde{\mathcal{O}}_{\lambda'}$ such that \eqref{eq36} commutes.
Since condition \eqref{fpf4} is satisfied for $\tilde{\mathcal{O}}_{\lambda'}$
and $\mathrm{G}$ is an equivalence, it follows that \eqref{fpf4} is 
satisfied for $\mathcal{O}_V$ as well. This shows that $\mathcal{O}_V$
is an fpf-category. 

The corresponding statement for the category $\mathcal{O}_{\overline{\lambda}}$ reduces to $\mathcal{O}_V$
using the adjointness of $E\otimes{}_-$ and $\check{E}^*\otimes{}_-$
as follows. We have
\begin{displaymath}
\mathrm{Hom}_{\mathfrak{g}}(E_1\otimes \Delta(V),
E_2\otimes \Delta(V))\cong
\mathrm{Hom}_{\mathfrak{g}}(\Delta(V),
\check{E}_1^*\otimes E_2\otimes \Delta(V))
\end{displaymath}
by adjointness and the homomorphism space on the right hand side 
can be computed inside $\mathcal{O}_V$. Similarly for the
morphisms of the corresponding projective functors. Now the claim follows
from the observation that the evaluation map from \eqref{fpf4}
commutes with adjunction.

The statement for the category 
$\mathcal{O}_{\overline{\lambda}}^{\mathfrak{p}}$ follows from that
for the category $\mathcal{O}_{\overline{\lambda}}$ since every 
homomorphism between projective modules in 
$\mathcal{O}_{\overline{\lambda}}^{\mathfrak{p}}$ comes from a
homomorphism between projective modules in
$\mathcal{O}_{\overline{\lambda}}$.

\subsection{Proof of Theorem~\ref{thm21}\eqref{thm21.1}
and \eqref{thm21.2}}\label{s4.5}

Let us first assume that $\mathfrak{p}=\mathfrak{h}\oplus \mathfrak{n}^+$.
If $\mathfrak{g}$ is one of the superalgebras
$\mathfrak{gl}(m,n)$, $\mathfrak{sl}(m,n)$,
$\mathfrak{osp}(m,n)$ or $\mathfrak{psl}(n,n)$, the
claim of Theorem~\ref{thm21}\eqref{thm21.1} is proved in
\cite[Proposition~5.1(ii)]{Go4}. The idea of the following proof for
type $\mathfrak{q}$-superalgebras was suggested by Maria Gorelik and 
follows closely \cite[Section~8]{Go4}. We are going to define some
analogue of Gorelik's notion of a {\em perfect mate}, show that
it exists for type $\mathfrak{q}$-superalgebras and use it
to prove our statement.

Denote by $\chi_{\lambda}$ the $U(\mathfrak{g})$-central character 
of $\Delta(V)$ and by $\chi_{\lambda}^0$ the $U(\mathfrak{g}_0)$-central 
character of $\Delta(\lambda)$. 

\begin{lemma}\label{lem41}
Let $L$ be a simple $U(\mathfrak{g})$-module with central character
$\chi_{\lambda}$. Then $\mathrm{Res}\,L$ has a non-zero component
with $U(\mathfrak{g}_0)$-central character $\chi_{\lambda}^0$.
\end{lemma}

\begin{proof}
Set $J:=\mathrm{Ann}_{U}L$. Then $J$ is a primitive ideal
of $U$ and hence, by \cite{Mu}, coincides with the
annihilator of some simple highest weight module $N$. Since
$L$ has central character $\chi_{\lambda}$, we can choose 
$N\in\mathcal{O}_V$. From \eqref{eq36} (with $\lambda'=\lambda$ by 
\cite{FM}) it follows that $\mathrm{Res}\,N$ has a non-zero component
with central character $\chi_{\lambda}^0$. Let 
$\chi_1=\chi_{\lambda}^0$, $\chi_2,\dots,\chi_k$ be all central characters
occurring in $\mathrm{Res}\,N$, and $\mathtt{m}_1,\dots,\mathtt{m}_k$
the corresponding maximal ideals in $Z(\mathfrak{g}_0)$. Let 
$l_1,\dots,l_k$ be minimal possible such that 
$\prod_{i=1}^k\mathtt{m}_i^{l_i}$ annihilates $N$. Then
$\prod_{i=2}^k\mathtt{m}_i^{l_i}$ does not annihilate $N$ and thus
it does not annihilate $L$ either. At the same time, the nonzero
space $\prod_{i=2}^k\mathtt{m}_i^{l_i} L$ is annihilated by 
$\mathtt{m}_1^{l_1}$. The claim follows.
\end{proof}

Denote $\mathtt{m}=\mathtt{m}_1$, where $\mathtt{m}_1$ is as in the
proof of Lemma~\ref{lem41}. Consider the set
\begin{displaymath}
X:=\{v\in \mathcal{L}(\Delta(\lambda),\Delta(\lambda)):
v\mathtt{m}^l=\mathtt{m}^l v=0,l\gg 0\}
\end{displaymath}
and define $B:=UXU$, which is a $U$-subbimodule of 
$\mathcal{L}(\Delta(\lambda),\Delta(\lambda))$.

\begin{proposition}\label{prop42}
We have $B=\mathcal{L}(\Delta(\lambda),\Delta(\lambda))$.
\end{proposition}

\begin{proof}
Let $B'$ denote the cokernel of the embedding
$B\hookrightarrow \mathcal{L}(\Delta(\lambda),\Delta(\lambda))$.
Then $B'$ satisfies
\begin{displaymath}
\{v\in B':v\mathtt{m}^l=\mathtt{m}^l v=0,l\gg 0\}=0
\end{displaymath}
by construction of $B$. The fact that $B'=0$ is now proved 
mutatis mutandis \cite[8.4.2]{Go4}.
\end{proof}

Let $M$ denote the $\mathfrak{g}_0$-submodule of $\Delta(V)$,
annihilated by $\mathtt{m}$. As $\lambda$ is generic by 
assumption, the module $M$ is isomorphic to a direct sum of
copies of $\Delta(\lambda)$. The direct summands of $M$ are
indexed by a basis of $V$ and hence the number of direct summands
equals $\dim V$. Due to Proposition~\ref{prop42}, 
to complete the proof of Theorem~\ref{thm21}\eqref{thm21.1} it 
is enough to show that $U$ surjects onto 
the $\mathfrak{g}_0$-bimodule 
\begin{displaymath}
\mathcal{L}(M,M)\cong
\mathrm{End}_{\mathbb{C}}(V)\otimes 
\mathcal{L}(\Delta(\lambda),\Delta(\lambda)).
\end{displaymath}
We know that $U(\mathfrak{g}_0)$ surjects onto 
$\mathcal{L}(\Delta(\lambda),\Delta(\lambda))$ (see 
\cite[6.9(10)]{Ja}). By \cite[A.3.2]{Go2}, under our assumption that 
$\Pi V\not\cong V$ the algebra $U(\mathfrak{h})$ surjects onto the 
matrix algebra $\mathrm{End}_{\mathbb{C}}(V)$. The claim follows.

For an arbitrary $\mathfrak{p}$ the claim of 
Theorem~\ref{thm21}\eqref{thm21.1} follows from the
case $\mathfrak{p}=\mathfrak{h}\oplus \mathfrak{n}^+$ similarly to
\cite[6.9]{Ja}.

The claim of Theorem~\ref{thm21}\eqref{thm21.2} follows from
Theorem~\ref{thm21}\eqref{thm21.1} mutatis mutandis 
\cite[Theorem~3.1]{MiSo}.

\subsection{Some conjectures following Theorem~\ref{thm21}}\label{s4.6}

In this subsection $\mathfrak{p}=\mathfrak{h}\oplus\mathfrak{n}_+$,
$\lambda\in\mathfrak{h}_0^*$, $V$ is a simple
$\mathfrak{h}_0$-module of weight $\lambda$ and 
$I=\mathrm{Ann}_{U}\Delta(V)$.

\begin{conjecture}\label{conj43}
{\rm 
Assume $\lambda$ is strongly typical and regular. Then,
the adjoint $\mathfrak{g}$-module $(U/I)^{\mathrm{ad}}$ is a direct
sum of injective finite-dimensional modules.
}
\end{conjecture}

By \cite[6.8(3)]{Ja}, for every finite dimensional
$\mathfrak{g}$-module $E$ there is a natural isomorphism
\begin{displaymath}
\mathrm{Hom}_{\mathfrak{g}}(E,
\mathcal{L}(\Delta(\lambda),\Delta(\lambda))^{\mathrm{ad}})
\cong\mathrm{Hom}_{\mathfrak{g}}(\Delta(\lambda),
\check{E}^*\otimes \Delta(\lambda)).
\end{displaymath}
Since the functor $\mathrm{Hom}_{\mathfrak{g}}(\Delta(\lambda),
({}_-)^*\otimes \Delta(\lambda))$ is exact on the category of 
finite dimensional $\mathfrak{g}$-modules (as tensoring over
$\mathbb{C}$ is exact and $\Delta(\lambda)$ is projective), 
it follows that the adjoint $\mathfrak{g}$-module
$\mathcal{L}(\Delta(\lambda),\Delta(\lambda))^{\mathrm{ad}}$ is a direct
sum of injective finite dimensional modules. In particular, 
Theorem~\ref{thm21}\eqref{thm21.1} implies Conjecture~\ref{conj43}
in the case $V\not\cong\Pi V$.

\begin{conjecture}\label{conj44}
{\rm 
Assume $\lambda$ is strongly typical and regular. Then the 
bimodule  $U/I$ is a direct summand of 
$\mathcal{L}(\Delta(\lambda),\Delta(\lambda))$.
}
\end{conjecture}

Conjecture~\ref{conj44} obviously implies Conjecture~\ref{conj43}.
Denote by $\varphi$ the automorphism of $U(\mathfrak{g})$,
which multiplies elements from $\mathfrak{g}_0$ with $1$ and 
elements from $\mathfrak{g}_1$ with $-1$.
For a $U$-bimodule $B$ denote by $B{}^{\varphi}$ the
bimodule, obtained by twisting the right action of $U$ with $\varphi$.

\begin{conjecture}\label{conj45}
{\rm 
Assume $\lambda$ is strongly typical and regular, and
$V\cong\Pi V$. Then we have the following bimodule
decomposition:
$\mathcal{L}(\Delta(\lambda),\Delta(\lambda))\cong
U/I\oplus (U/I)^{\varphi}$.
}
\end{conjecture}

From \cite[A.3.2]{Go2} it follows that in the case $V\cong\Pi V$
the algebra $U(\mathfrak{h})$ does not surject onto the 
matrix algebra $\mathrm{End}_{\mathbb{C}}(V)$, in fact, the image of
$U(\mathfrak{h})$ in $\mathrm{End}_{\mathbb{C}}(V)$ has dimension
$\frac{1}{2}\dim\mathrm{End}_{\mathbb{C}}(V)$. It follows that in this 
case the image of $U$ in $\mathcal{L}(\Delta(\lambda),\Delta(\lambda))$
is only ``one half'' of $\mathcal{L}(\Delta(\lambda),\Delta(\lambda))$.

\subsection{Annihilators and Kostant's problem}\label{s4.7}

In this subsection we work under the assumptions of 
Theorem~\ref{thm21}\eqref{thm21.1}.
Due to the equivalences from \cite{Go3} and \cite{FM}, the module
$\Delta^{\mathfrak{p}}(V)$ has simple socle, say $K$, which is 
the unique simple subquotient of $\Delta^{\mathfrak{p}}(V)$ of maximal
GK-dimension. Thanks to \eqref{eq36}, for every projective functor $\theta$ 
we have
\begin{equation}\label{eq44}
\dim \mathrm{Hom}_{\mathfrak{g}}
(\theta \Delta^{\mathfrak{p}}(V),\Delta^{\mathfrak{p}}(V)) =
\dim \mathrm{Hom}_{\mathfrak{g}}(\theta K,K). 
\end{equation} 

\begin{proposition}\label{prop48}
\begin{enumerate}[$($a$)$]
\item\label{prop48.1} The natural inclusion
$U/\mathrm{Ann}_U(K)\hookrightarrow \mathcal{L}(K,K)$ is surjective.
\item\label{prop48.2} $\mathrm{Ann}_U(K)=
\mathrm{Ann}_U(\Delta^{\mathfrak{p}}(V))$.
\item\label{prop48.3} We have $\mathrm{Ann}_U(N)=
\mathrm{Ann}_U(\Delta^{\mathfrak{p}}(V))$ for any nonzero submodule
$N$ of $\Delta^{\mathfrak{p}}(V)$.
\end{enumerate}
\end{proposition}

\begin{proof}
From \eqref{eq44} we have $\mathcal{L}(K,K)\cong
\mathcal{L}(\Delta^{\mathfrak{p}}(V),\Delta^{\mathfrak{p}}(V))$
using the same argument as in \cite[Lemma~11]{KM}. Now
\eqref{prop48.1} follows from the fact that $U$ surjects onto 
$\mathcal{L}(\Delta^{\mathfrak{p}}(V),\Delta^{\mathfrak{p}}(V))$,
established in Theorem~\ref{thm21}\eqref{thm21.1}.
Since $\mathrm{Ann}_U(\Delta^{\mathfrak{p}}(V))\subset
\mathrm{Ann}_U(K)$, claim \eqref{prop48.2} follows from
\eqref{prop48.1}. As $K$ is the simple socle of 
$\Delta^{\mathfrak{p}}(V)$, claim \eqref{prop48.3} follows from
\eqref{prop48.2}.
\end{proof}

\begin{corollary}\label{cor49}
All Verma modules in $\mathcal{O}_V$ have the same annihilator
and for each such module $N$ we have 
$U/\mathrm{Ann}_U(N)\cong \mathcal{L}(N,N)$.
\end{corollary}

\begin{proof}
The first claim follows from Proposition~\ref{prop48}\eqref{prop48.3} 
since all Verma modules in $\mathcal{O}_V$ are submodules of $\Delta(V)$.

To prove the second claim we need some notation. Let $W$ denote the Weyl
group. Then Verma modules in $\mathcal{O}_V$ are naturally indexed
by elements of $W$. For $w\in W$ let $\Delta(w)$ denote the 
corresponding Verma module (with the convention $\Delta(e)=\Delta(V)$).
Since all $\Delta(w)$ have the same annihilator by the first claim,
to prove the second claim it is enough to show that for any simple
reflection $s\in W$ we have
\begin{equation}\label{eq45}
\mathcal{L}(\Delta(w),\Delta(w))=
\mathcal{L}(\Delta(sw),\Delta(sw)).
\end{equation}
Without loss of generality we may assume that 
$\Delta(sw)\hookrightarrow \Delta(w)\tto C$, where $C$ is just the
cokernel. Then $C$ is $s$-finite while the simple top and the
simple socle of $\Delta(sw)$ are $s$-infinite. Using this,
\eqref{eq45} is proved similarly to \cite[Lemma~11]{KM}. The claim follows.
\end{proof}

\subsection{Serre functor for category 
$\mathcal{O}_{\overline{\lambda}}$}\label{s4.8}

Fix one representative in every isomorphism class of indecomposable 
projective objects in $\mathcal{O}^{\mathfrak{p}}_{\overline{\lambda}}$ 
and denote by $\mathcal{C}^{\mathfrak{p}}_{\overline{\lambda}}$  
the full subcategory of $\mathcal{O}^{\mathfrak{p}}_{\overline{\lambda}}$ 
which these fixed objects generate. The category
$\mathcal{C}^{\mathfrak{p}}_{\overline{\lambda}}$ is an slf-category and
$\mathcal{O}^{\mathfrak{p}}_{\overline{\lambda}}$ is equivalent 
to $\mathcal{C}^{\mathfrak{p}}_{\overline{\lambda}}\text{-}\mathrm{mod}$
(note that $\mathcal{C}^{\mathfrak{p}}_{\overline{\lambda}}\cong (\mathcal{C}^{\mathfrak{p}}_{\overline{\lambda}})^{\mathrm{op}}$ 
because of $\star$). Now we are ready to formulate our main result.

\begin{theorem}\label{thm52}
Let $\lambda\in \mathfrak{h}_0^*$ be strongly typical, regular,
dominant and generic. Let further $V$ be a simple
$\mathfrak{h}_0$-module of weight $\lambda$ and assume that 
$\Pi V\not\cong V$. Set $\Delta:=\Delta^{\mathfrak{p}}(V)$. Then we have:
\begin{enumerate}[$($a$)$]
\item\label{thm52.1} For any $P,N\in
\mathcal{O}^{\mathfrak{p}}_{\overline{\lambda}}$ with $P$ projective
there is an isomorphism
\begin{displaymath}
\mathrm{Hom}_{\mathfrak{g}}\big(N,
\mathcal{L}(P,\Delta)^{\circledast}\otimes_{U}
\Delta\big)\cong \mathrm{Hom}_{\mathfrak{g}}(P,N)^*,
\end{displaymath}
natural in both $P$ and $N$.
\item\label{thm52.2} The left derived of the functor
$\mathcal{L}({}_-,\Delta)^{\circledast}\otimes_{U}\Delta$
is a Serre functor on 
$\cP(\mathcal{C}^{\mathfrak{p}}_{\overline{\lambda}})$. 
\end{enumerate}
\end{theorem}

\begin{proof}
Using Theorem~\ref{thm21}\eqref{thm21.2} we will prove 
claim \eqref{thm52.1} along the lines of the proof of Theorem~\ref{thm4}.
In fact, we need only to prove an analogue of Proposition~\ref{prop55},
the rest of the proof is identical to that of Theorem~\ref{thm4}. 

Our first observation is that it is enough to prove 
Proposition~\ref{prop55} under the assumption that $N$ is projective.
The case of general $N$ reduces to the case of projective $N$ by taking
the first two steps of the projective resolution (because both sides
of \eqref{eq56} are right exact in $N$). 

If $N$ is projective, then $N\cong \theta\Delta$ for some projective
functor $\theta$. Similarly to Subsection~\ref{s4.6} one shows
that the adjoint $\mathfrak{g}$-module
$\mathcal{L}(\Delta,\theta\Delta)^{\mathrm{ad}}$ is a direct sum
of injective finite dimensional modules. 
Let $L_{\mathfrak{g}}$ and $L_{\mathfrak{g}_0}$ denote the trivial
$\mathfrak{g}$- and $\mathfrak{g}_0$-modules, respectively.
Note that $\mathrm{Res}\,L_{\mathfrak{g}}\cong L_{\mathfrak{g}_0}$.
Let $I(0)$ denote the injective hull of $L_{\mathfrak{g}}$ in the 
category of finite dimensional modules.

\begin{lemma}\label{lem91}
The top of $I(0)$ is isomorphic to the trivial module.
\end{lemma}

\begin{proof}
The assumption  $\Pi V\not\cong V$ is equivalent to the assumption
that $\mathrm{Ind}$ is isomorphic to coinduction. By adjunction we have
\begin{displaymath}
\mathrm{Hom}_{\mathfrak{g}}(
L_{\mathfrak{g}},\mathrm{Ind}\,L_{\mathfrak{g}_0})
\cong
\mathrm{Hom}_{\mathfrak{g}_0}(\mathrm{Res}\,L_{\mathfrak{g}},
L_{\mathfrak{g}_0})\neq 0,
\end{displaymath}
which implies that $I(0)$ is a direct summand  
of $\mathrm{Ind}\,L_{\mathfrak{g}_0}$. Similarly, 
\begin{displaymath}
\mathrm{Hom}_{\mathfrak{g}}(\mathrm{Ind}\,
L_{\mathfrak{g}_0},L_{\mathfrak{g}})
\cong
\mathrm{Hom}_{\mathfrak{g}_0}(L_{\mathfrak{g}_0},
\mathrm{Res}\,L_{\mathfrak{g}})\neq 0
\end{displaymath}
and hence $L_{\mathfrak{g}}$ appears in the top of
$\mathrm{Ind}\,L_{\mathfrak{g}_0}$.

At the same time, we claim that $L_{\mathfrak{g}}$ does not appear 
in the top of any other indecomposable injective module $I$.
Indeed, if $L$ is the simple socle of $I$, then $L\neq L_{\mathfrak{g}}$ 
and hence $L$ has a nontrivial simple $\mathfrak{g}_0$-submodule, say $N_{\mathfrak{g}_0}$. Then, by adjunction, 
\begin{displaymath}
\mathrm{Hom}_{\mathfrak{g}}(L,\mathrm{Ind}\,N_{\mathfrak{g}_0})
\cong
\mathrm{Hom}_{\mathfrak{g}_0}(\mathrm{Res}\,L,N_{\mathfrak{g}_0})\neq 0
\end{displaymath}
and hence $I$ appears as a direct summand in 
$\mathrm{Ind}\,N_{\mathfrak{g}_0}$. 
At the same time, again by adjunction,
\begin{displaymath}
\mathrm{Hom}_{\mathfrak{g}}(\mathrm{Ind}\,
N_{\mathfrak{g}_0},L_{\mathfrak{g}})
\cong
\mathrm{Hom}_{\mathfrak{g}_0}(N_{\mathfrak{g}_0},
\mathrm{Res}\,L_{\mathfrak{g}})=0.
\end{displaymath}
The claim follows.
\end{proof}

In the case $P=\Delta$ we have $\mathcal{L}(\Delta,\Delta)=U/I$
by Theorem~\ref{thm21}\eqref{thm21.1}. It follows that in this
case
\begin{equation}\label{eq92}
\mathrm{Hom}_{\mathfrak{g}\text{-}\mathfrak{g}}
(\mathcal{L}(\Delta,P),\mathcal{L}(\Delta,N))\cong
\mathcal{L}(\Delta,N)^{\mathfrak{g}}_0.
\end{equation}
Similarly to the proof of Lemma~\ref{lem5}, we obtain that 
\begin{equation}\label{eq93}
\mathcal{L}(P,\Delta)
\otimes_{U\text{-}U}\mathcal{L}(\Delta,N)\cong
((\mathcal{L}(\Delta,N)^{\circledast})^{\mathfrak{g}}_0)^*.
\end{equation}
As $\mathcal{L}(\Delta,N)^{\mathrm{ad}}$ is a direct sum of indecomposable
injective modules and $I(0)$ has isomorphic top and socle, the 
multiplicities of $L_{\mathfrak{g}}$ in the top and the socle of 
$\mathcal{L}(\Delta,N)^{\mathrm{ad}}$ coincide. This implies that
the vector spaces in \eqref{eq92} and \eqref{eq93} have the same
dimension, say $k$, giving us an analogue of Lemma~\ref{lem5}.

Let $Y$ be the
isotypic component of $I(0)$ in $\mathcal{L}(\Delta,N)^{\mathrm{ad}}$.
It is easy to see that
$\mathcal{K}\mathcal{L}(\Delta,N)$ contains both the complement $Y'$ of
$Y$ and the radical $\mathrm{Rad}\, Y$ of $Y$. The dimension argument 
implies that $\mathcal{K}\mathcal{L}(\Delta,N)$ coincides with
$Y'\oplus \mathrm{Rad}\, Y$. This means that 
$\mathcal{L}(\Delta,N)/\mathcal{K}\mathcal{L}(\Delta,N)$
is the top of $Y$. Clearly, 
$\mathcal{L}(\Delta,N)^{\mathfrak{g}}_0$ is the socle of
$Y$.
Let $u\in U(\mathfrak{g})$ be any element which annihilates the
radical of $I(0)$ but not $I(0)$. Applying $u$ provides a unique
(up to a nonzero scalar) isomorphism from the top of $Y$ to the socle 
of $Y$. This gives us an analogue of the isomorphism \eqref{eq58}. 
The rest of the proof of Proposition~\ref{prop55} carries over 
analogously.  Claim \eqref{thm52.1} follows.

Claim \eqref{thm52.2} follows directly from \eqref{thm52.1}.
This completes the proof.
\end{proof}

\subsection{Applications}\label{s4.9}

In this subsection we work under the assumptions of Theorem~\ref{thm52}.

\begin{corollary}\label{cor51}
The Nakayama functor on $\mathcal{O}^{\mathfrak{p}}_{\overline{\lambda}}$ 
naturally commutes with projective functors.
\end{corollary}

\begin{proof}
Mutatis mutandis Corollary~\ref{cor8}.
\end{proof}

Consider the endofunctor $\mathrm{C}$ of 
$\mathcal{O}^{\mathfrak{p}}_{\overline{\lambda}}$
of partial coapproximation with respect to projective-injective modules. 

\begin{corollary}\label{cor52}
The Nakayama functor on $\mathcal{O}^{\mathfrak{p}}_{\overline{\lambda}}$ 
is isomorphic to $\mathrm{C}^2$.
\end{corollary}

\begin{proof}
Mutatis mutandis Corollary~\ref{cor9}.
\end{proof}

Fix one representative in every isomorphism class of indecomposable
projective-injective objects in 
$\mathcal{O}^{\mathfrak{p}}_{\overline{\lambda}}$ and denote by 
$\mathcal{P}^{\mathfrak{p}}_{\overline{\lambda}}$ the full 
subcategory of 
$\mathcal{O}^{\mathfrak{p}}_{\overline{\lambda}}$ generated
by these fixed objects. The category
$\mathcal{P}^{\mathfrak{p}}_{\overline{\lambda}}$
is an slf-category.

\begin{corollary}\label{cor54}
The category $\mathcal{P}^{\mathfrak{p}}_{\overline{\lambda}}$
is symmetric.
\end{corollary}

\begin{proof}
Mutatis mutandis Corollary~\ref{cor11}.
\end{proof}

As a special case of Corollary~\ref{cor54} we obtain that the algebra
$\mathcal{P}^{\mathfrak{g}}_{\overline{\lambda}}$ describing
finite dimensional modules in 
$\mathcal{O}_{\overline{\lambda}}$ is symmetric. For the superalgebra
$\mathfrak{gl}(m,n)$ this is proved in \cite{BS}. 
In the case $\Pi V\cong V$ this is no longer true in general
(it is easy to see that the projective cover and the injective hull
of the trivial $\mathfrak{q}(1)$-module are not isomorphic).
Instead we propose the following:

\begin{conjecture}\label{conj55}
{\rm
Under the assumption $\Pi V\cong V$ the functor $\Pi$ is the Serre 
functor on $\cP(\mathcal{P}^{\mathfrak{g}}_{\overline{\lambda}})$.
} 
\end{conjecture}

\section{Finite dimensional $\mathfrak{q}(2)$-modules}\label{s5}

\subsection{The result}\label{s5.1}

In this section we use our results to describe all blocks of
the category of  finite dimensional $\mathfrak{q}(2)$-modules.

\begin{theorem}\label{thm61}
Every block of the category of finite dimensional 
$\mathfrak{q}(2)$-modules is equivalent to the category of 
finite dimensional modules over one of the following algebras 
given by quiver and relations:
\begin{enumerate}[(a)]
\item\label{thm61.1} 
\begin{displaymath}
\bullet 
\end{displaymath}
\item\label{thm61.1a} 
\begin{displaymath}
\quad\quad\quad\quad\quad\quad\quad\quad\quad
\xymatrix{ 
\bullet\ar@(dr,ur)[]_{a} 
}\quad\quad\quad\,\,\, a^2=0.
\end{displaymath}
\item\label{thm61.2} 
\begin{displaymath}
\xymatrix{ 
\bullet\ar@/^/[rr]^{a} &&\bullet\ar@/^/[rr]^{a}\ar@/^/[ll]^{b} 
&& \bullet\ar@/^/[rr]^{a}\ar@/^/[ll]^{b}&&\dots\ar@/^/[ll]^{b}
}\quad a^2=b^2=0, ab=ba.
\end{displaymath}
\item\label{thm61.3} 
\begin{displaymath}
\xymatrix{ 
\bullet\ar[rr]^{c}\ar@/^/[d]^{h} &&\bullet\ar@/^/[rr]^{a}\ar[lld]^{d}
&& \bullet\ar@/^/[rr]^{a}\ar@/^/[ll]^{b}&&\dots\ar@/^/[ll]^{b}\\
\bullet\ar[rr]_{c}\ar@/^/[u]^{h} &&\bullet\ar@/^/[rr]^{a}\ar[llu]_{d} 
&& \bullet\ar@/^/[rr]^{a}\ar@/^/[ll]^{b}&&\dots\ar@/^/[ll]^{b}
}
\begin{array}{l}
\\\\a^2=b^2=0, ab=ba;\\chd=ba, dch=hdc;\\h^2=ac=db=cd=0.\\
\end{array}
\end{displaymath}
\end{enumerate}
\end{theorem}

It is worth emphasizing that the algebra in 
Theorem~\ref{thm61}\eqref{thm61.3} is not quadratic (not even homogeneous) 
and hence not Koszul either (in contrast to algebras from \cite{BS}).
Note that all algebras in Theorem~\ref{thm61} are special
biserial, in particular, they are tame.

\subsection{Notation and the typical case}\label{s5.2}

We fix the standard matrix unit basis 
$\{e_{11},e_{12},e_{21},e_{22},
\overline{e}_{11},\overline{e}_{12},
\overline{e}_{21},\overline{e}_{22}\}$, where the overlined elements
are odd.

Write $\mathfrak{q}(2)$-weights
in the form $\lambda=(\lambda_1,\lambda_2)\in\mathbb{C}^2$ with respect to 
the dual basis of $\mathfrak{h}_0$. 
Simple finite dimensional $\mathfrak{q}(2)$-modules
are determined uniquely up to parity change by their highest weight.
If $\lambda$ is the highest weight of
a simple finite dimensional $\mathfrak{q}(2)$-module $L$, then 
either $\lambda=0$ or $\lambda_1-\lambda_2\in\mathbb{N}$.
The module $L$ satisfies $L\cong \Pi L$ if and only if exactly one of
$\lambda_1$ and $\lambda_2$ equals zero. 
Denote by $\mathcal{W}$ the set of highest weights of
all simple finite dimensional $\mathfrak{q}(2)$-modules.
We will loosely denote these modules by $L(\lambda)$ 
and $\Pi L(\lambda)$ (but $L(0)$ is the trivial module).

The weight $\lambda\in \mathcal{W}$ is atypical
if and only if it is of the form $(k,-k)$ for some 
$k\in\{0,\frac{1}{2},1,\frac{3}{2},2,\dots\}$. The facts that 
strongly typical blocks are semi-simple (and hence described by
Theorem~\ref{thm61}\eqref{thm61.1}) and that other typical blocks 
are described by Theorem~\ref{thm61}\eqref{thm61.1a}
follow from \cite[6.2]{Fr}.

\subsection{The non-integral atypical block}\label{s5.3}

For $k\in\{0,\frac{1}{2},1,\frac{3}{2},2,\dots\}$
set $\lambda^k=(k,-k)$ and
let $N(\lambda^k)$ denote the simple $\mathfrak{g}_0$-module
with highest weight $\lambda^k$. It has dimension $2k+1$.
We denote by $P(\lambda^k)$ the projective cover of $L(\lambda^k)$
(in the category of finite dimensional modules).

Recall from \cite{PS,Ma2} that simple atypical 
$\mathfrak{q}(2)$-modules look as follows: we have the
trivial module $L(0)$, its parity changed $\Pi L(0)$, and for 
every $k\in\{\frac{1}{2},1,\frac{3}{2},2,\dots\}$
the even part of both  $L(\lambda^k)$ and $\Pi L(\lambda^k)$ is 
isomorphic to $N(\lambda^k)$. 
In particular, using adjunction we obtain that the module 
$\mathrm{Ind}\,N(\lambda^0)$ is indecomposable, while 
$\mathrm{Ind}\,N(\lambda^k)\cong P(\lambda^k)\oplus\Pi P(\lambda^k)$
for all $k\neq 0$.

For any $\mathfrak{g}_0$-module $M$ we have
$\mathrm{Ind}\,M\cong \bigwedge\mathfrak{g}_1\otimes M$ as 
$\mathfrak{g}_0$-modules. The $\mathfrak{g}_0$-module
$\bigwedge\mathfrak{g}_1$ is a direct sum of four copies of the
trivial module and four copies of $N(\lambda^1)$. Moreover, the
odd part of  $\bigwedge\mathfrak{g}_1$ is isomorphic to the even part.

Assume now that $k\in\{\frac{1}{2},\frac{3}{2},\frac{5}{2},\dots\}$.
In this case the even part of $P(\lambda^k)$ is isomorphic, as
a $\mathfrak{g}_0$-module, to the module 
\begin{displaymath}
(N(0)\oplus N(\lambda^1))
\otimes N(\lambda^k)\cong
\begin{cases}
N(\lambda^k)\oplus N(\lambda^{k+1})\oplus
N(\lambda^k), & k=\frac{1}{2};\\
N(\lambda^k)\oplus N(\lambda^{k-1})\oplus N(\lambda^{k+1})\oplus
N(\lambda^k), & \text{otherwise}. 
\end{cases}
\end{displaymath}
This implies that $P(\lambda^k)$ has length three if $k=\frac{1}{2}$
and length four otherwise. By Corollary~\ref{cor54}, 
$P(\lambda^k)$ has isomorphic top and socle, which are thus 
isomorphic to $L(\lambda^k)$. The other composition factor of
$P(\lambda^{\frac{1}{2}})$ has highest weight $\lambda^\frac{3}{2}$.
For $k\neq \frac{1}{2}$, the other two composition factors of
$P(\lambda^{k})$ have highest weight $\lambda^{k+1}$
and $\lambda^{k-1}$. Interchanging $L(\lambda^{k})$ and 
$\Pi L(\lambda^{k})$ for some $k$, if necessary, we may assume that 
all composition factors of $P(\lambda^{k})$ have the form
$L(\lambda^{k})$ or $L(\lambda^{k\pm 1})$.  This means that we
have the following Loewy filtrations of projective modules:
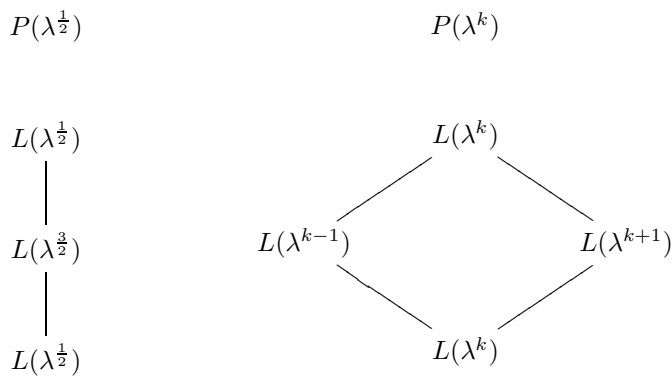
\begin{displaymath}
\xymatrix{P(\lambda^{\frac{1}{2}})\\
L(\lambda^{\frac{1}{2}})\ar@{-}[d]\\
L(\lambda^{\frac{3}{2}})\ar@{-}[d]\\L(\lambda^{\frac{1}{2}})
}\quad \quad\quad \quad \quad \quad  
\xymatrix{&P(\lambda^{k})&\\&L(\lambda^{k})\ar@{-}[dl]\ar@{-}[dr]&\\
L(\lambda^{k-1})\ar@{-}[dr]&&L(\lambda^{k+1})\ar@{-}[dl]\\&L(\lambda^{k})&
}\quad \quad 
\end{displaymath}
It is now easy to see that the full Serre 
subcategory generated by $L(\lambda^{k})$,
$k\in\{\frac{1}{2},\frac{3}{2},\frac{5}{2},\dots\}$, forms a block
which is equivalent to the category of finite dimensional modules 
over the algebra given in Theorem~\ref{thm61}\eqref{thm61.2}.

\subsection{The principal block}\label{s5.4}

Similarly to the previous subsection one shows that for 
$k\in\{2,3,4,\dots\}$ the projective module $P(\lambda^{k})$
has length four, isomorphic simple top and socle, and two 
other composition factors which may be chosen to be isomorphic to
$L(\lambda^{k-1})$ and $L(\lambda^{k+1})$. So we need only to
determine the structure of $P(0)$ and $P(\lambda^{1})$.

A similar argument for $P(\lambda^{1})$ shows that it has length
five with top and socle isomorphic to $L(\lambda^{1})$ and
the other three composition factors isomorphic to 
$L(0)$, $\Pi L(0)$ and $L(\lambda^{2})$.

By a character argument, the module $P(0)$ has length six with top 
and socle isomorphic to  $L(0)$, two other composition factors isomorphic to 
$\Pi L(0)$, and the two remaining composition factors having highest
weight $\lambda^1$. Since $P(\lambda^{1})$ contains both
$L(0)$ and  $\Pi L(0)$, the module $\Pi P(\lambda^{1})$
contains both $L(0)$ and $\Pi L(0)$ as well. This implies that 
both $L(\lambda^{1})$ and $\Pi L(\lambda^{1})$ must appear in the
injective hull of $L(0)$. Since we work with a symmetric algebra,
it follows that $P(0)$ contains both $L(\lambda^{1})$ and 
$\Pi L(\lambda^{1})$, that is we now know the composition factors
of both $P(0)$ and $P(\lambda^{1})$.

\begin{lemma}\label{lem71}
We have $\mathrm{Ext}^1_{\mathfrak{g}}(L(0),\Pi L(0))\cong \mathbb{C}$. 
\end{lemma}

\begin{proof}
The elements $e_{12},e_{21},\overline{e}_{12}$ and $\overline{e}_{21}$ 
obviously annihilate any extension of $L(0)$ by $\Pi L(0)$.
An extension of $L(0)$ by $\Pi L(0)$ is thus given by specifying two 
scalars $\alpha$ and $\beta$ which represent the action of 
$\overline{e}_{11}$ and $\overline{e}_{22}$, respectively.
These scalars must satisfy $\alpha-\beta=0$ as
$[\overline{e}_{12},e_{21}]=\overline{e}_{11}-\overline{e}_{22}$.
On the other hand, it is straightforward to verify that
any $\alpha$ and $\beta$ satisfying $\alpha-\beta=0$ give rise to
a nontrivial extension of $L(0)$ by $\Pi L(0)$. This proves the lemma.
\end{proof}

From Lemma~\ref{lem71} it follows that $P(0)$ has Loewy length at
least four for otherwise the top of $\mathrm{Rad}\,P(0)$ would contain two
copies of $\Pi L(0)$. Therefore the top  of $\mathrm{Rad}\,P(0)$
has at most two composition factors, one of which must be isomorphic to
$\Pi L(0)$. The parity change argument implies that the top of
$\mathrm{Rad}\,P(0)$ must contain more than $\Pi L(0)$. Without loss of
generality we may assume that it contains $L(\lambda^1)$.
Since we already know that $P(\lambda^1)$ does not contain any
$\Pi L(\lambda^1)$ and that $L(0)$ does not extend $L(0)$, it follows that
$P(0)$ has the composition structure as shown on the left part of
the following picture:
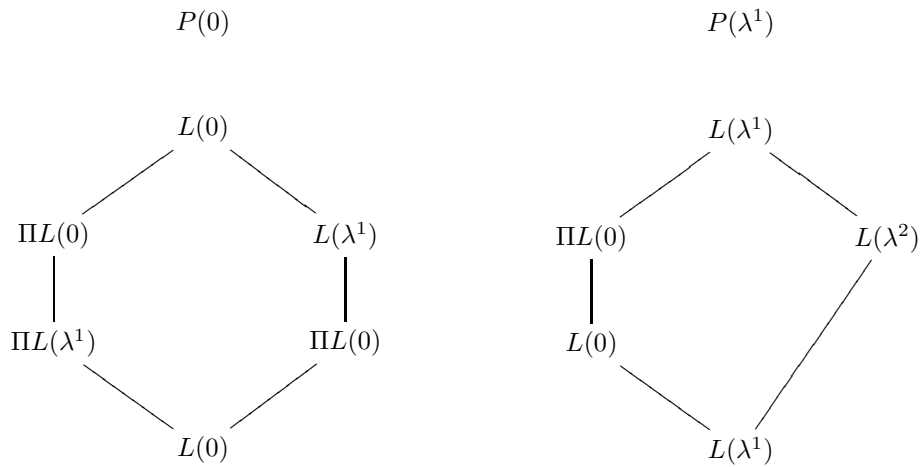
\begin{displaymath}
\xymatrix{&P(0)&\\
&L(0)\ar@{-}[dr]\ar@{-}[dl]&\\
\Pi L(0)\ar@{-}[d]&&L(\lambda^1)\ar@{-}[d]\\
\Pi L(\lambda^1)\ar@{-}[dr]&&\Pi L(0)\ar@{-}[dl]\\
&L(0)&
}\quad \quad\quad \quad \quad \quad  
\xymatrix{&P(\lambda^{1})&\\
&L(\lambda^{1})\ar@{-}[dl]\ar@{-}[dr]&\\
\Pi L(0)\ar@{-}[d]&&L(\lambda^{2})\ar@{-}[ddl]\\
L(0)\ar@{-}[dr]&&\\
&L(\lambda^{1})&\\
}\quad \quad 
\end{displaymath}
We see that the image of the unique up to scalar nonzero map from 
$P(\lambda^{1})$ to $P(0)$ contains an extension of 
$\Pi L(0)$ by $L(0)$. This implies that the composition structure
of $P(\lambda^{1})$ is as shown on the right hand side of the
above picture. Now it is easy to see that the full Serre
subcategory generated by $L(\lambda^{k})$ and $\Pi L(\lambda^{k})$,
$k\in\{0,1,2,\dots\}$ forms a block
which is equivalent the category of finite dimensional modules 
over the algebra given in Theorem~\ref{thm61}\eqref{thm61.3}.
Since this is the block containing the trivial module, it is called
the {\em principal block}.

\vspace{1cm}

\noindent
Volodymyr Mazorchuk, Department of Mathematics, Uppsala University,
Box 480, 751 06, Uppsala, SWEDEN, {\tt mazor\symbol{64}math.uu.se};
http://www.math.uu.se/$\tilde{\hspace{1mm}}$mazor/.
\vspace{0.1cm}

\noindent
Vanessa Miemietz, School of Mathematics, University of East Anglia,
Norwich, UK, NR4 7TJ, {\tt v.miemietz\symbol{64}uea.ac.uk};
http://www.uea.ac.uk/$\tilde{\hspace{1mm}}$byr09xgu/.

\end{document}